\documentclass[10pt,aps,prb, preprint,reqno]{amsart}
\usepackage{amsmath}
\usepackage{amssymb}

\newtheorem{theorem}{Theorem}[section]
\newtheorem{remark}{Remark}[section]
\newtheorem{lemma}[theorem]{Lemma}
\newtheorem{proposition}[theorem]{Proposition}

\begin{document}
\title[4D Navier-Stokes equations]{Regularity criterion of the 4D Navier-Stokes equations involving two velocity field components}
\author{Kazuo Yamazaki}  
\date{}
\maketitle

\begin{abstract}
We study the Serrin-type regularity criteria for the solutions to the four-dimensional Navier-Stokes equations and magnetohydrodynamics system. We show that the sufficient condition for the solution to the four-dimensional Navier-Stokes equations to preserve its initial regularity for all time may be reduced from a bound on the four-dimensional velocity vector field to any two of its four components, from a bound on the gradient of the velocity vector field to the gradient of any two of its four components, from a gradient of the pressure scalar field to any two of its partial derivatives. Results are further generalized to the magnetohydrodynamics system. These results may be seen as a four-dimensional extension of many analogous results that exist in the three-dimensional case and also component reduction results of many classical results. 

\vspace{5mm}

\textbf{Keywords: Navier-Stokes equations, Magnetohydrodynamics system, regularity criteria}
\end{abstract}
\footnote{2000MSC : 35B65, 35Q35, 35Q86}
\footnote{Department of Mathematics, Washington State University, Pullman, WA 99164-3113}

\section{Introduction}

We study the $N$-dimensional $(N \geq 2)$ Navier-Stokes equations (NSE) and magnetohydrodynamics (MHD) system defined respectively as follows: 
\begin{subequations}
\begin{align}
&\frac{du}{dt} + (u\cdot\nabla) u + \nabla \pi = \nu \Delta u,\\
&\nabla \cdot u = 0, \hspace{3mm} u(x,0) = u_{0}(x),
\end{align}
\end{subequations}
\begin{subequations}
\begin{align}
&\frac{du}{dt} + (u\cdot\nabla) u + \nabla \pi = \nu \Delta u + (b\cdot\nabla) b,\\
&\frac{db}{dt} + (u\cdot\nabla) b = \eta \Delta b + (b\cdot\nabla)u,\\
&\nabla \cdot u = \nabla\cdot b = 0, \hspace{3mm} (u,b)(x,0) = (u_{0}, b_{0})(x),
\end{align}
\end{subequations}
where $u = (u_{1}, \hdots, u_{N}): \mathbb{R}^{N} \times \mathbb{R}^{+} \mapsto \mathbb{R}^{N}, b = (b_{1}, \hdots, b_{N}): \mathbb{R}^{N} \times \mathbb{R}^{+} \mapsto \mathbb{R}^{N}, \pi: \mathbb{R}^{N} \times \mathbb{R}^{+} \mapsto \mathbb{R}$ represent the velocity vector field, magnetic vector field and pressure scalar field respectively. We denote by the parameters $\nu, \eta \geq 0$ the viscosity and magnetic diffusivity respectively. Hereafter, we also denote $\frac{d}{dt}$ by $\partial_{t}$ and  $\frac{d}{dx_{i}}$ by $\partial_{i}, i = 1, \hdots, N$ and by $\nabla_{i,j}$ the gradient vector field with $\partial_{i}, \partial_{j}$ on the $i$-th, $j$-th component respectively and zero elsewhere and $\Delta_{i,j}$ the sum of second derivatives in the $i$-th and $j$-th directions , e.g. $\nabla_{1,2} = (\partial_{1}, \partial_{2}, 0, \hdots, 0), \Delta_{1,2} = \sum_{k=1}^{2}\partial_{kk}^{2}.$ 

The importance and difficulty of the global regularity issue of the solution to these two systems are well known. In short, this is because the systems are both energy-supercritical in any dimension bigger than two even with $\nu, \eta > 0$. Indeed, e.g. for the MHD system, taking $L^{2}$-inner products with $(u,b)$ on (2a)-(2b) respectively and integrating in time lead to 
\begin{equation}
\sup_{t \in [0,T]} (\lVert u\rVert_{L^{2}}^{2} + \lVert b\rVert_{L^{2}}^{2})(t) + \int_{0}^{T} \lVert \nabla u\rVert_{L^{2}}^{2} + \lVert \nabla b\rVert_{L^{2}}^{2} d\tau \leq \lVert u_{0} \rVert_{L^{2}}^{2} + \lVert b_{0} \rVert_{L^{2}}^{2}.
\end{equation}
On the other hand, it can be shown that if $(u,b)(x,t)$ solves the system (2a)-(2c), then so does $(u_{\lambda}, b_{\lambda})(x,t) \triangleq \lambda (u,b)(\lambda x, \lambda^{2} t)$. A direct computation shows that
\begin{equation*}
\lVert u_{\lambda}(x,t) \rVert_{L^{2}}^{2} + \lVert b_{\lambda}(x,t) \rVert_{L^{2}}^{2} = \lambda^{2-N}(\lVert u(x, \lambda^{2} t) \rVert_{L^{2}}^{2} + \lVert b(x, \lambda^{2} t)\rVert_{L^{2}}^{2}).
\end{equation*}
We call an equation with a scaling symmetry critical when the strongest norm for which an \emph{a priori } estimate is available is scaling-invariant. Thus, it is standard to classify the two-dimensional NSE and the MHD system as energy-critical while for any dimension higher, energy-supercritical; in fact, it can be considered that the supercriticality increases in dimension. 

In two-dimensional case with $\nu, \eta > 0$, the authors in [22, 26] have shown the uniqueness of the solution to the NSE and the MHD system respectively. In fact, in the two-dimensional case due to the simplicity of the form after taking curls, when the dissipative and diffusive terms are replaced by fractional Laplacians, their powers may be reduced furthermore below one; we refer interested readers to [34] for the NSE with $\nu = 0$, [6] and references found therein for the MHD system. In any dimension strictly higher than two, the problem concerning the global regularity of the strong solution and the uniqueness of the weak solution to both systems remain open and hence much effort has been devoted to provide criterion so that they hold. We now review some of them, emphasizing on those of most relevance to the current manuscript. 

Initiated by the author in [27], it has been established that if a weak solution $u$ of the NSE with $\nu > 0$ satisfies
\begin{equation}
u \in L^{r}(0, T; L^{p}(\mathbb{R}^{N})), \hspace{3mm} \frac{N}{p} + \frac{2}{r} \leq 1, \hspace{3mm} p \in (N, \infty],
\end{equation}
then $u$ is smooth (see [9, 11] for the endpoint case). In [2], the author showed that if $u$ solves the NSE (1a)-(1b) with $\nu > 0$ and 
\begin{equation}
\nabla u \in L^{r}(0, T; L^{p}(\mathbb{R}^{N})), \hspace{2mm} N \geq 3, \hspace{2mm} \frac{N}{p} + \frac{2}{r} = 2, \hspace{2mm} 1 < r \leq \min\{2, \frac{N}{N-2}\},
\end{equation}
then $u$ is a regular solution. For the MHD system, the authors in [15, 37] independently showed that the sufficient condition for the regularity of the solution pair $(u,b)$ to the MHD system (2a)-(2c) may be reduced to just $u$. For many more important results in this direction of research, all of which we cannot list here, we refer to the prominent work of [1, 14] and references found therein. We do mention that the author in [38] showed that only in case $N = 3, 4$, $u$, the solution to the NSE (1a)-(1b) with $\nu > 0$, is regular and unique if  
\begin{equation}
\nabla \pi \in L^{r}(0,T; L^{p}(\mathbb{R}^{N})), \hspace{3mm} \frac{N}{p} + \frac{2}{r} \leq 3, \hspace{3mm} \frac{N}{3} \leq p \leq \infty. 
\end{equation}
We emphasize that the norm $\lVert \cdot \rVert_{L_{T}^{r}L_{x}^{p}}$ in (4) is scaling invariant precisely when $\frac{N}{p} + \frac{2}{r} = 1$; i.e. 
\begin{equation*}
\int_{0}^{T} \lVert u_{\lambda}(x,t) \rVert_{L^{p}}^{r} dt = \int_{0}^{\lambda^{2} T} \lVert u(x,t)\rVert_{L^{p}}^{r} dt \hspace{1mm} \text{ if and only if } \hspace{1mm} \frac{N}{p} + \frac{2}{r} = 1,
\end{equation*}
where $u_{\lambda}(x,t) = \lambda u(\lambda x, \lambda^{2} t)$, and similarly for the norm in (5) at the endpoint of $2$. 

We now survey some component reduction results of such criterion. The authors in [20] showed that if $u$ solves the NSE with $N = 3, \nu > 0$ and 
\begin{align}
u_{3} \in& L^{r}(0, T; L^{p}(\mathbb{R}^{3})), \hspace{3mm} \frac{3}{p} + \frac{2}{r} \leq \frac{5}{8}, \hspace{3mm} r \in [\frac{54}{23}, \frac{18}{5}],\\
\text{ or } \nabla u_{3} \in& L^{r}(0, T; L^{p}(\mathbb{R}^{3})), \hspace{3mm} \frac{3}{p} + \frac{2}{r} \leq \frac{11}{6}, \hspace{3mm} r \in [\frac{24}{5}, \infty],\nonumber
\end{align}
then the solution is regular (see also [3, 39] for similar results on $u_{3}, \nabla u_{3}$).  For the MHD system, in particular the authors in [17] showed that if $(u,b)$ solves (2a)-(2c) with $N = 3, \nu, \eta > 0$ and 
\begin{equation}
u_{3}, b \in L^{r}(0, T; L^{p}(\mathbb{R}^{3})), \hspace{3mm} \frac{3}{p} + \frac{2}{r} \leq \frac{3}{4} + \frac{1}{2p}, \hspace{3mm} p > \frac{10}{3},
\end{equation}
then the solution pair $(u,b)$ remains smooth for all time. In [29], the author reduced this constraint on $u_{3}, b$ to $u_{3}, b_{1}, b_{2}$ in special cases making use of the special structure of (2b). For more interesting component reduction results of the regularity criterion, we refer to e.g. [4, 5, 12, 16, 21, 24, 28, 30, 36]. In particular, the authors in [7] obtained a regularity criterion for the three-dimensional NSE in terms of only $u_{3}$ in a scaling-invariant norm, although no longer $L_{T}^{r}L_{x}^{p}$-space (see also [33]). In relevance to our discussion below, we already emphasize that every component reduction result listed here is of the case $N = 3$. 

We now motivate the study of (1a)-(1b), (2a)-(2c) in fourth dimension specifically. It has been realized by many mathematicians working in the research direction of the NSE that the dimension four deserves special attention (see e.g. Section 4 [18]). The significance of the fourth dimension for the NSE (and six-dimensional stationary NSE) has motivated much investigation in the research direction of partial regularity theory (see e.g. [8, 10, 25]); we also recall (6) which holds only for $N = 3, 4$. In fact, fourth dimension being a certain threshold to the component reduction regularity criteria can be seen clearly as follows. To the best of the author's knowledge, all such component reduction results to the systems (1a)-(1b) and (2a)-(2c) are obtained through an $H^{1}$-estimate. Due to Lemma 2.3, higher regularity follows once we show that the solution e.g. $u$ in the case of the NSE (1a)-(1b) satisfies $\int_{0}^{T} \lVert \nabla u\rVert_{L^{N}(\mathbb{R}^{N})}^{2}d\tau < \infty$. This implies that because $H^{1}(\mathbb{R}^{N}) \hookrightarrow L^{N}(\mathbb{R}^{N})$ only for $N  =2, 3, 4$ but not $N > 4$ by Sobolev embedding, $H^{1}$-bound, from which $u\in L^{2}(0, T; H^{2}(\mathbb{R}^{N}))$ follows from the dissipative term, is sufficient for higher regularity only if $N = 2, 3, 4$. Thus, in dimension strictly higher than four, one needs to bound beyond $H^{1}$-norm; however, because the decomposition of the non-linear terms is the most important ingredient of component reduction results (see Proposition 3.1), this will complicate the proof significantly. To the best of the author's knowledge, component reduction results for dimension strictly larger than three does not exist in the literature. 

Let us also discuss the two major obstacles in extending the component reduction results of regularity criteria from dimension three to four. In the case of the NSE (1a)-(1b) with $N = 3, \nu > 0$, the standard procedure to obtain a criteria in terms of $u_{3}$ may be to, e.g. first estimate every partial derivative except the last and hence $\lVert \nabla_{1,2} u\rVert_{L^{2}}$ and in this process separate $u_{3}$ in the non-linear term:
\begin{equation}
\int (u\cdot\nabla) u \cdot\Delta_{1,2} u dx \leq c \int \lvert u_{3} \rvert \lvert \nabla u\rvert \lvert \nabla\nabla_{1,2} u\rvert dx
\end{equation}
where $\nabla_{1,2} = (\partial_{1}, \partial_{2}, 0), \Delta_{1,2} = \sum_{k=1}^{2}\partial_{kk}^{2}$ (cf. [20] Lemma 2.3). Thereafter, upon a full gradient and hence an $H^{1}$-estimate, on the non-linear term one separates $\lvert \nabla_{1,2} u\rvert$: 
\begin{equation}
\int (u\cdot\nabla) u \cdot\Delta u dx \leq c\int \lvert \nabla_{1,2} u\rvert \lvert \nabla u\rvert^{2} dx
\end{equation}
(cf. [39]) so that the $\lVert \nabla_{1,2} u\rVert_{L^{2}}$-estimate may be applied in (10). In the case $N  =4$, it seems difficult to separate $u_{3}$ or even $u_{3}$ and $u_{4}$ in $\int (u\cdot\nabla) u \cdot \Delta_{1,2,3}u dx$. Our first key observation is that we can separate $u_{3}, u_{4}$ from $\int (u\cdot\nabla) u \cdot \Delta_{1,2}u dx$ (See Proposition 3.1). However, this leaves two other directions, namely $x_{3}, x_{4}$, instead of only one in contrast to the case $N = 3$ and disables us to obtain an inequality analogous to (10) upon the full $H^{1}$-estimate due to a sum of this type: 
\begin{align*}
\sum_{j=1}^{4}\sum_{i,k=3}^{4}\int \partial_{k}u_{i}\partial_{i}u_{j}\partial_{k}u_{j}  dx
\end{align*}
(see (43)). We observe that in the three-dimensional case, $i,j$ and $k$ sum up to only 3 so that using $\nabla\cdot u = 0$ from (1b), one may deduce 
\begin{align*}
\sum_{j=1}^{3}\sum_{i,k=3}^{3}\int \partial_{k}u_{i}\partial_{i}u_{j}\partial_{k}u_{j} dx =& \sum_{j=1}^{3} \int \partial_{3}u_{3}\partial_{3}u_{j}\partial_{3}u_{j} dx \\
=& -\sum_{j=1}^{3} \int (\partial_{1}u_{1} + \partial_{2}u_{2})\partial_{3}u_{j}\partial_{3}u_{j}dx
\end{align*}
and hence (10) follows. However, in the four-dimensional case, there are cross-terms such as $\partial_{3}u_{4}$ which disables us to reach (10). Our second key observation is that the non-linear term may be seen as an operator as a sum of 
\begin{equation*}
u\cdot\nabla = \sum_{i=1}^{4}u_{i}\partial_{i} = \sum_{i=1}^{2}u_{i}\partial_{i} + \sum_{i=3}^{4}u_{i}\partial_{i}
\end{equation*}
so that in the first sum, the $\nabla_{1,2}$-estimate may be applied while in the second, use our hypothesis on $u_{3}, u_{4}$ (see (43) and also (46)).  

We now present our results: 
\begin{theorem}
Let $N = 4$ and  
\begin{equation}
u \in C([0,T); H^{s}(\mathbb{R}^{4})) \cap L^{2}([0,T); H^{s+1}(\mathbb{R}^{4}))
\end{equation}
be the solution to the NSE (1a)-(1b) for a given $u_{0} \in H^{s}(\mathbb{R}^{4}), s > 4$. Suppose $u_{3}, u_{4}$ with their corresponding $p_{i}, r_{i}, i = 3, 4$ satisfy the following roles of $f$: 
\begin{equation}
\int_{0}^{T} \lVert f\rVert_{L^{p_{i}}}^{r_{i}} d\tau \leq c, \hspace{3mm} \frac{4}{p_{i}} + \frac{2}{r_{i}} \leq \frac{1}{p_{i}} + \frac{1}{2}, \hspace{3mm} 6 < p_{i} \leq \infty, 
\end{equation}
or $\sup_{t\in [0, T]}\lVert f(t) \rVert_{L^{6}}$ being sufficiently small. Then $u$ remains in the same regularity class (11) on $[0, T']$ for some $T'> T$. 
\end{theorem}

\begin{theorem}
Let $N = 4$ and $u$ in the regularity class of (11) be the solution to the NSE (1a)-(1b) for a given $u_{0} \in H^{s}(\mathbb{R}^{4}), s > 4$. Suppose $\nabla u_{3}, \nabla u_{4}$ with their corresponding $p_{i}, r_{i}, i = 3, 4$ satisfy the following roles of $f$:  
\begin{equation}
\int_{0}^{T} \lVert f\rVert_{L^{p_{i}}}^{r_{i}} d\tau \leq c, \hspace{3mm} \frac{4}{p_{i}} + \frac{2}{r_{i}} \leq \begin{cases}
\frac{5}{4} + \frac{1}{p_{i}},  &\text{ if } \frac{12}{5} < p_{i} \leq 4\\
1 + \frac{2}{p_{i}},  &\text{ if } 4 < p_{i} \leq \infty
\end{cases}, \hspace{3mm} 
\end{equation}
or $\sup_{t\in [0, T]}\lVert f(t) \rVert_{L^{\frac{12}{5}}}$ being sufficiently small. Then $u$ remains in the same regularity class (11) on $[0, T']$ for some $T'> T$. 

\end{theorem}

\begin{theorem}
Let $N = 4$ and  
\begin{equation}
u, b \in C([0,T); H^{s}(\mathbb{R}^{4})) \cap L^{2}([0,T); H^{s+1}(\mathbb{R}^{4}))
\end{equation}
be the solution pair to the MHD system (2a)-(2c) for a given $u_{0}, b_{0} \in H^{s}(\mathbb{R}^{4}), s > 4$. Suppose $u_{3}, u_{4}, b$ with their corresponding $p_{i}, r_{i}, i = 3, 4, b$ satisfy the following roles of $f$: 
\begin{equation}
\int_{0}^{T} \lVert f\rVert_{L^{p_{i}}}^{r_{i}} d\tau \leq c, \hspace{3mm} \frac{4}{p_{i}} + \frac{2}{r_{i}} \leq \frac{1}{p_{i}} + \frac{1}{2}, \hspace{3mm} 6 < p_{i} \leq \infty, 
\end{equation}
or $\sup_{t\in [0, T]}\lVert f(t) \rVert_{L^{6}}$ being sufficiently small. Then $u,b$ remain in the same regularity class (14) on $[0, T']$ for some $T'> T$. 
\end{theorem}

\begin{theorem}
Let $N = 4$ and $u,b$ in the regularity class of (14) be the solution pair to the MHD system (2a)-(2c) for a given $u_{0}, b_{0} \in H^{s}(\mathbb{R}^{4}), s > 4$. Suppose $\nabla u_{3}, \nabla u_{4}, \nabla b$ with their corresponding $p_{i}, r_{i}, i = 3, 4, b$ satisfy the following roles of $f$: 
\begin{equation}
\int_{0}^{T} \lVert f\rVert_{L^{p_{i}}}^{r_{i}} d\tau \leq c, \hspace{1mm} \frac{4}{p_{i}} + \frac{2}{r_{i}} \leq \begin{cases}
\frac{5}{4} + \frac{1}{p_{i}},  &\text{ if } \frac{12}{5} < p_{i} \leq 4\\
1 + \frac{2}{p_{i}},  &\text{ if } 4 < p_{i} \leq \infty
\end{cases}, \hspace{1mm} 
\end{equation}
or $\sup_{t\in [0, T]}\lVert f (t)\rVert_{L^{\frac{12}{5}}}$ being sufficiently small. Then $u,b$ remain in the same regularity class (14) on $[0, T']$ for some $T'> T$. 
\end{theorem}

\begin{theorem}
Let $N = 4$ and $u$ in the regularity class of (11) be the solution to the NSE (1a)-(1b) for a given $u_{0} \in H^{s}(\mathbb{R}^{4}), s > 4$. Suppose $\partial_{3} \pi, \partial_{4}\pi$ with their corresponding $p_{i}, r_{i}, i = 3, 4$ satisfy the following roles of $f$: 
\begin{equation}
\int_{0}^{T} \lVert f\rVert_{L^{p_{i}}}^{r_{i}} d\tau \leq c, \hspace{3mm} \frac{4}{p_{i}} + \frac{2}{r_{i}} < \frac{8}{3}, \hspace{3mm} \frac{12}{7} < p_{i} < 6.
\end{equation}
Then $u$ remains in the same regularity class (11) on $[0, T']$ for some $T'> T$. 
\end{theorem}

\begin{remark}
\begin{enumerate}
\item In comparison of Theorem 1.1 with (4), Theorem 1.2 with (5), Theorem 1.5 with (6), we may consider the results of this manuscript as component reduction of many previous work. Moreover, in comparison of Theorems 1.1 and 1.2 with (7), Theorem 1.3 with (8), we may consider the results of this manuscript as four-dimension extension of many previous work in three-dimension. 
\item The Lemma 2.3 of [20] has found much applications, e.g. in the study on the anisotropic NSE (e.g. [35]). We note that our Proposition 3.1 can be readily generalized further to any $\mathbb{R}^{N}, N \geq 3$; we chose to state the case $N =4$ for the simplicity of presentation. 
\item In [32], the author showed that  for dimensions $N = 3, 4, 5$, $N$-many component regularity criteria may be reduced to $(N-1)$ many components for the generalized MHD system following the method in [28]; the results in [32] and this manuscript do not cover each other. In [31] the author also obtained a regularity criteria of $N$-dimensional porous media equation governed by Darcy's law in terms of one partial derivative of the scalar-valued solution. The method in [31] cannot be applied to  (1a)-(1b), (2a)-(2c). 
\end{enumerate}
\end{remark}

In the Preliminaries section, we set up notations and state key facts. Local theory is well-known (cf. [23]); hence, by the standard argument of continuation of local theory, we only need to obtain $H^{s}$-bounds. We present the proofs of Theorems 1.3, 1.4 and 1.5. Because the NSE is the MHD system at $b \equiv 0$, the proofs of Theorem 1.3 and 1.4 immediately deduce Theorems 1.1 and 1.2 respectively. Thereafter, we conclude with a brief further discussion. 

\section{Preliminaries}

Throughout the rest of the manuscript, we shall assume $\nu, \eta = 1$ for simplicity. For brevity, we write $\int f$ for $\int_{\mathbb{R}^{N}} f(x) dx$ and $A \lesssim_{a,b} B$ when there exists a constant $c \geq 0$ of significant dependence only on $a, b$ such that $A \leq c B$, similarly $A \approx_{a,b} B$ in case $A = cB$. We denote the fractional Laplacian operator $\Lambda^{s} \triangleq (-\Delta)^{\frac{s}{2}}$ and 
\begin{align*}
&W(t) \triangleq (\lVert \nabla_{1,2} u\rVert_{L^{2}}^{2} + \lVert \nabla_{1,2}b\rVert_{L^{2}}^{2})(t), \hspace{8mm} X(t) \triangleq (\lVert \nabla u\rVert_{L^{2}}^{2} + \lVert \nabla b\rVert_{L^{2}}^{2})(t),\\
&Y(t) \triangleq (\lVert \nabla\nabla_{1,2} u\rVert_{L^{2}}^{2} + \lVert \nabla\nabla_{1,2} b\rVert_{L^{2}}^{2})(t) , \hspace{3mm} Z(t) \triangleq (\lVert \Delta u\rVert_{L^{2}}^{2} + \lVert \Delta b\rVert_{L^{2}}^{2})(t).
\end{align*}
The following is a special case of Troisi's inequality (cf. [13]). The proof of the case $N = 3$ in the Appendix of [5] can be readily generalized to the case $N = 4$: 

\begin{lemma}
Let $f \in C_{0}^{\infty} (\mathbb{R}^{4})$. Then 
\begin{equation}
\lVert f\rVert_{L^{4}} \lesssim \lVert \partial_{1} f\rVert_{L^{2}}^{\frac{1}{4}} \lVert \partial_{2} f\rVert_{L^{2}}^{\frac{1}{4}} \lVert \partial_{3} f\rVert_{L^{2}}^{\frac{1}{4}} \lVert \partial_{4} f\rVert_{L^{2}}^{\frac{1}{4}}.
\end{equation}
\end{lemma}
We will use the following elementary inequality frequently: 
\begin{equation}
(a+b)^{p} \leq 2^{p}(a^{p} + b^{p}), \hspace{3mm} \text{for } 0 \leq p < \infty \text{ and } a, b \geq 0.
\end{equation}

We will also use the following commutator estimate to prove another lemma concerning higher regularity: 
\begin{lemma}
(cf. [19]) Let $f,g$ be smooth such that $\nabla f \in L^{p_{1}}, \Lambda^{s-1}g \in L^{p_{2}}, \Lambda^{s}f \in L^{p_{3}}, g \in L^{p_{4}}, p \in (1,\infty), \frac{1}{p} = \frac{1}{p_{1}}+\frac{1}{p_{2}} = \frac{1}{p_{3}} + \frac{1}{p_{4}}, p_{2}, p_{3} \in (1, \infty), s > 0.$ Then 
\begin{equation*}
\lVert \Lambda^{s}(fg) - f\Lambda^{s}g\rVert_{L^{p}} \lesssim (\lVert \nabla f\rVert_{L^{p_{1}}}\lVert \Lambda^{s-1}g\rVert_{L^{p_{2}}} + \lVert \Lambda^{s}f\rVert_{L^{p_{3}}}\lVert g\rVert_{L^{p_{4}}}).
\end{equation*}
\end{lemma}
An immediate application of Lemma 2.2 gives the following result:
\begin{lemma}
Let $(u,b)$ be the solution to the MHD system (2a)-(2c) in $[0,T]$ with $u_{0}, b_{0} \in H^{s}(\mathbb{R}^{N}), N \geq 3, s > 2 + \frac{N}{2}$. Then if $\int_{0}^{T} \lVert \nabla u\rVert_{L^{N}}^{2} + \lVert \nabla b\rVert_{L^{N}}^{2} d\tau \lesssim 1$, then 
\begin{equation*}
\sup_{t\in [0,T]} (\lVert \Lambda^{s}u\rVert_{L^{2}}^{2} + \lVert \Lambda^{s} b\rVert_{L^{2}}^{2})(t) + \int_{0}^{T} \lVert \Lambda^{s}\nabla  u\rVert_{L^{2}}^{2} + \lVert \Lambda^{s} \nabla b\rVert_{L^{2}}^{2}d\tau \lesssim 1. 
\end{equation*}
\end{lemma}

\begin{proof}

This is a standard computation; we sketch it for completeness. We apply $\Lambda^{s}$ on (2a)-(2b), take $L^{2}$-inner products with $\Lambda^{s}u, \Lambda^{s} b$ respectively to obtain
\begin{align*}
& \frac{1}{2} \partial_{t} (\lVert \Lambda^{s} u\rVert_{L^{2}}^{2} + \lVert \Lambda^{s} b\rVert_{L^{2}}^{2}) + \lVert \Lambda^{s} \nabla u\rVert_{L^{2}}^{2} + \lVert \Lambda^{s} \nabla b\rVert_{L^{2}}^{2}\\
=& -\int [\Lambda^{s} ((u\cdot\nabla) u) - u\cdot\nabla \Lambda^{s} u]\cdot\Lambda^{s} u - \int [\Lambda^{s} ((u\cdot\nabla) b) - u\cdot\nabla \Lambda^{s} b] \cdot\Lambda^{s} b\\
&+ \int [\Lambda^{s} ((b\cdot\nabla) b) - b\cdot\nabla \Lambda^{s} b]\cdot\Lambda^{s} u + \int [\Lambda^{s} ((b\cdot\nabla) u) - b\cdot\nabla \Lambda^{s} u] \cdot\Lambda^{s} b\\
\lesssim &(\lVert \nabla u\rVert_{L^{N}} + \lVert \nabla b\rVert_{L^{N}})(\lVert \Lambda^{s} u\rVert_{L^{2}} + \lVert \Lambda^{s} b\rVert_{L^{2}}) (\lVert \Lambda^{s} \nabla u\rVert_{L^{2}} + \lVert \Lambda^{s} \nabla b\rVert_{L^{2}})\\
\leq& \frac{1}{2}(\lVert \Lambda^{s} \nabla u\rVert_{L^{2}}^{2} + \lVert \Lambda^{s} \nabla b\rVert_{L^{2}}^{2}) + c(\lVert \nabla u\rVert_{L^{N}}^{2} + \lVert \nabla b\rVert_{L^{N}}^{2})(\lVert \Lambda^{s} u\rVert_{L^{2}}^{2} + \lVert \Lambda^{s} b\rVert_{L^{2}}^{2})
\end{align*}
by H$\ddot{o}$lder's inequalities, Lemma 2.2, Sobolev embedding of $\dot{H}^{1}(\mathbb{R}^{N}) \hookrightarrow L^{\frac{2N}{N-2}}(\mathbb{R}^{N})$, Young's inequalities and (19). Thus, after absorbing, Gronwall's inequality completes the proof of Lemma 2.3.
\end{proof}

Due to Lemma 2.3,  the proof of our theorems are complete once we obtain $H^{1}$-bound. 

\section{Proof of Theorem 1.3}

\subsection{$\lVert \nabla_{1,2} u\rVert_{L^{2}}^{2} + \lVert \nabla_{1,2} b\rVert_{L^{2}}^{2}$-estimate}

We first prove an important decomposition which we present as a proposition:
\begin{proposition}
Let $N = 4$ and $(u,b)$ be the solution pair to the MHD system (2a)-(2c). Then 
\begin{align}
&\int (u\cdot\nabla) u \cdot \Delta_{1,2} u +  (u\cdot\nabla) b \cdot \Delta_{1,2} b -  (b\cdot\nabla) b \cdot \Delta_{1,2} u - (b\cdot\nabla) u \cdot \Delta_{1,2} b\nonumber \\
\lesssim& \int (\lvert u_{3} \rvert + \lvert u_{4}\rvert) \lvert \nabla u\rvert \lvert \nabla\nabla_{1,2} u\rvert + \lvert b\rvert (\lvert \nabla u\rvert + \lvert \nabla b\rvert)(\lvert \nabla\nabla_{1,2} u\rvert + \lvert \nabla\nabla_{1,2} b\rvert).
\end{align}
Moreover, 
\begin{align}
&\int (u\cdot\nabla) u \cdot \Delta_{1,2} u + (u\cdot\nabla) b \cdot \Delta_{1,2} b - (b\cdot\nabla) b \cdot \Delta_{1,2} u -  (b\cdot\nabla) u \cdot \Delta_{1,2} b \nonumber \\
\lesssim& \int (\lvert \nabla u_{3} \rvert + \lvert \nabla u_{4} \rvert) \lvert \nabla_{1,2} u\rvert \lvert \nabla u\rvert + \lvert \nabla b\rvert \lvert \nabla_{1,2} b\rvert \lvert \nabla u\rvert.
\end{align}
\end{proposition}

\begin{proof}
We write components-wise and integrate by parts to obtain 
\begin{align}
&\int (u\cdot\nabla) u \cdot\Delta_{1,2} u\\
=& -\sum_{i,j=1}^{4}\sum_{k=1}^{2}\int\partial_{k}u_{i}\partial_{i}u_{j}\partial_{k}u_{j}\nonumber\\
=& -\sum_{j=1}^{4}\sum_{i,k=1}^{2}\int\partial_{k}u_{i}\partial_{i}u_{j}\partial_{k}u_{j} - \sum_{i=3}^{4}\sum_{j=1}^{4}\sum_{k=1}^{2}\int\partial_{k}u_{i}\partial_{i}u_{j}\partial_{k}u_{j}\nonumber\\
=& -\sum_{i,j,k=1}^{2}\int\partial_{k}u_{i}\partial_{i}u_{j}\partial_{k}u_{j} - \sum_{j=3}^{4}\sum_{i,k=1}^{2}\int\partial_{k}u_{i}\partial_{i}u_{j}\partial_{k}u_{j} - \sum_{i=3}^{4}\sum_{j=1}^{4}\sum_{k=1}^{2}\int\partial_{k}u_{i}\partial_{i}u_{j}\partial_{k}u_{j}.\nonumber
\end{align}
For the second and third integrals of (22), we integrate by parts to obtain 
\begin{align}
&-\sum_{j=3}^{4}\sum_{i,k=1}^{2}\int\partial_{k}u_{i}\partial_{i}u_{j}\partial_{k}u_{j} - \sum_{i=3}^{4}\sum_{j=1}^{4}\sum_{k=1}^{2}\int\partial_{k}u_{i}\partial_{i}u_{j}\partial_{k}u_{j}\\
=& \sum_{j=3}^{4}\sum_{i,k=1}^{2} \int u_{j} \partial_{i}(\partial_{k}u_{i}\partial_{k}u_{j}) + \sum_{i=3}^{4}\sum_{j=1}^{4}\sum_{k=1}^{2} \int u_{i} \partial_{k}(\partial_{i}u_{j}\partial_{k}u_{j})\nonumber\\
\lesssim& \int (\lvert u_{3}\rvert + \lvert u_{4}\rvert) \lvert \nabla u\rvert \lvert \nabla\nabla_{1,2} u\rvert.\nonumber
\end{align}
On the other hand, we write the first integral of (22) explicitly 
\begin{align}
&-\sum_{i,j,k=1}^{2}\int \partial_{k}u_{i}\partial_{i}u_{j}\partial_{k}u_{j}\\
=& -\int(\partial_{1}u_{1})^{3} + \partial_{2}u_{1}\partial_{1}u_{1}\partial_{2}u_{1} + \partial_{1}u_{1}\partial_{1}u_{2}\partial_{1}u_{2} + \partial_{2}u_{1}\partial_{1}u_{2}\partial_{2}u_{2}\nonumber\\
&+ \partial_{1}u_{2} \partial_{2}u_{1}\partial_{1}u_{1} + \partial_{2}u_{2}\partial_{2}u_{1}\partial_{2}u_{1} + \partial_{1}u_{2}\partial_{2}u_{2}\partial_{1}u_{2} + (\partial_{2}u_{2})^{3} \triangleq \sum_{i=1}^{8} I_{i}.\nonumber
\end{align}
We combine and use the incompressibility condition of $u$ to obtain 
\begin{align}
&I_{1} + I_{8} = -\int (\partial_{1}u_{1})^{3} + (\partial_{2}u_{2})^{3}\\
=& \int (\partial_{1}u_{1})^{2}\partial_{2}u_{2} + (\partial_{1}u_{1})^{2} (\partial_{3}u_{3} + \partial_{4}u_{4})+ (\partial_{2}u_{2})^{2}\partial_{1}u_{1} + (\partial_{2}u_{2})^{2}(\partial_{3}u_{3} + \partial_{4}u_{4}).\nonumber
\end{align}
We combine the first and third terms to obtain 
\begin{align*}
\int (\partial_{1}u_{1})^{2} \partial_{2}u_{2} + (\partial_{2}u_{2})^{2} \partial_{1}u_{1} 
=- \int \partial_{1}u_{1}\partial_{2}u_{2}(\partial_{3}u_{3} + \partial_{4}u_{4}) 
\end{align*}
so that we may continue (25) by 
\begin{align}
I_{1} + I_{8} =& -\int \partial_{1}u_{1}\partial_{2}u_{2}(\partial_{3}u_{3} + \partial_{4}u_{4})\\
&+ \int (\partial_{1}u_{1})^{2} (\partial_{3}u_{3} + \partial_{4}u_{4}) + (\partial_{2}u_{2})^{2}(\partial_{3}u_{3} + \partial_{4}u_{4})\nonumber\\
=& \int u_{3} \partial_{3} (\partial_{1}u_{1}\partial_{2}u_{2}) + u_{4}\partial_{4}(\partial_{1}u_{1}\partial_{2}u_{2})\nonumber\\
&- \int u_{3} \partial_{3}[(\partial_{1}u_{1})^{2} + (\partial_{2}u_{2})^{2}] + u_{4} \partial_{4} [(\partial_{1}u_{1})^{2} + (\partial_{2}u_{2})^{2}]\nonumber\\
\lesssim& \int (\lvert u_{3}\rvert + \lvert u_{4}\rvert) \lvert \nabla u\rvert \lvert \nabla\nabla_{1,2} u\rvert.\nonumber
\end{align}
Similarly, 
\begin{align}
I_{2} + I_{6} =& -\int \partial_{2}u_{1}\partial_{1}u_{1}\partial_{2}u_{1} + \partial_{2}u_{2}\partial_{2}u_{1}\partial_{2}u_{1}\\
=& \int (\partial_{2}u_{1})^{2} (\partial_{3}u_{3} + \partial_{4}u_{4}) \lesssim \int (\lvert u_{3}\rvert + \lvert u_{4}\rvert) \lvert \nabla u\rvert \lvert \nabla\nabla_{1,2}u\rvert,\nonumber
\end{align}
\begin{align}
I_{3} + I_{7} =& -\int \partial_{1}u_{1}\partial_{1}u_{2}\partial_{1}u_{2} + \partial_{1}u_{2}\partial_{2}u_{2}\partial_{1}u_{2}\\
=& \int (\partial_{1}u_{2})^{2} (\partial_{3}u_{3} + \partial_{4}u_{4}) \lesssim \int (\lvert u_{3}\rvert + \lvert u_{4}\rvert) \lvert \nabla u\rvert \lvert \nabla\nabla_{1,2} u\rvert,\nonumber
\end{align}
\begin{align}
I_{4} + I_{5} =& -\int \partial_{2}u_{1} \partial_{1}u_{2} \partial_{2}u_{2} + \partial_{1}u_{2}\partial_{2}u_{1} \partial_{1}u_{1}\\
=& \int \partial_{2}u_{1} \partial_{1}u_{2} (\partial_{3}u_{3} + \partial_{4}u_{4}) \lesssim \int (\lvert u_{3} \rvert + \lvert u_{4} \rvert) \lvert \nabla u\rvert \lvert \nabla \nabla_{1,2} u\rvert.\nonumber
\end{align}
Next, we may estimate the other three terms as follows:
\begin{align}
&\int (u\cdot\nabla) b \cdot \Delta_{1,2} b -  (b\cdot\nabla) b \cdot \Delta_{1,2} u -  (b\cdot\nabla) u \cdot \Delta_{1,2} b\\
=& -\sum_{i,j=1}^{4}\sum_{k=1}^{2}\int \partial_{k}u_{i} \partial_{i}b_{j}\partial_{k}b_{j} + \sum_{i,j=1}^{4}\sum_{k=1}^{2}\int \partial_{k}b_{i}\partial_{i}b_{j}\partial_{k}u_{j} + \partial_{k}b_{i}\partial_{i}u_{j}\partial_{k}b_{j}\nonumber\\
=& \sum_{i,j=1}^{4}\sum_{k=1}^{2} \int \partial_{k}u_{i}b_{j} \partial_{ik}^{2} b_{j} - \sum_{i,j=1}^{4} \sum_{k=1}^{2} \int \partial_{k}b_{i} b_{j} \partial_{ik}^{2}u_{j} + b_{i}\partial_{k}(\partial_{i}u_{j}\partial_{k}b_{j})\nonumber\\
\lesssim& \int \lvert b\rvert (\lvert \nabla u\rvert + \lvert \nabla b\rvert) (\lvert \nabla\nabla_{1,2} u\rvert + \lvert \nabla\nabla_{1,2} b\rvert).\nonumber
\end{align}
Applying (26)-(29) in (24), considering (22), (23) and (30) we obtain (20). Now we go back to (22) and estimate the second and third integrals by 
\begin{align}
&-\sum_{j=3}^{4} \sum_{i,k=1}^{2} \int \partial_{k}u_{i}\partial_{i}u_{j} \partial_{k}u_{j} - \sum_{i=3}^{4} \sum_{j=1}^{4}\sum_{k=1}^{2}\int \partial_{k}u_{i}\partial_{i}u_{j}\partial_{k}u_{j}\\
\lesssim& \sum_{j=3}^{4}\sum_{k=1}^{2}\int \lvert \partial_{k} u \rvert \lvert \nabla u_{j} \rvert \lvert \partial_{k} u\rvert + \sum_{i=3}^{4}\sum_{k=1}^{2}\int \lvert \nabla u_{i} \rvert \lvert \nabla u\rvert \lvert \partial_{k} u\rvert \nonumber\\
\lesssim& \int (\lvert \nabla u_{3} \rvert + \lvert \nabla u_{4} \rvert) \lvert \nabla_{1,2} u\rvert \lvert \nabla u\rvert\nonumber
\end{align}
whereas continuing from (26), 
\begin{align}
I_{1} + I_{8} =& -\int \partial_{1}u_{1} \partial_{2}u_{2} (\partial_{3} u_{3} + \partial_{4} u_{4}) + \left((\partial_{1}u_{1})^{2} + (\partial_{2}u_{2})^{2}\right) (\partial_{3}u_{3} + \partial_{4} u_{4})\\
\lesssim& \int \lvert \nabla_{1,2} u\rvert^{2} ( \lvert \partial_{3}u_{3} \rvert + \lvert \partial_{4}u_{4}\rvert),\nonumber
\end{align}
continuing from (27), 
\begin{align}
I_{2} + I_{6} = \int (\partial_{2}u_{1})^{2}(\partial_{3}u_{3} + \partial_{4}u_{4}) \lesssim \int \lvert \nabla_{1,2} u\rvert^{2}(\lvert \partial_{3}u_{3} \rvert + \lvert \partial_{4}u_{4}\rvert),
\end{align}
continuing from (28),
\begin{align}
I_{3} + I_{7} = \int (\partial_{1}u_{2})^{2}(\partial_{3} u_{3} + \partial_{4}u_{4})
\lesssim \int \lvert \nabla_{1,2} u\rvert^{2}(\lvert \partial_{3}u_{3} \rvert + \lvert \partial_{4}u_{4}\rvert),
\end{align}
and continuing from (29),
\begin{align}
I_{4} + I_{5} = \int \partial_{2}u_{1}\partial_{1}u_{2}(\partial_{3}u_{3} + \partial_{4}u_{4})
\lesssim \int \lvert \nabla_{1,2} u\rvert^{2}(\lvert \partial_{3}u_{3} \rvert + \lvert \partial_{4}u_{4}\rvert).
\end{align}
Thus, considering (31)-(35) in (22),  we have shown 
\begin{equation}
\int (u\cdot\nabla) u \cdot \Delta_{1,2} u \lesssim \int (\lvert \nabla u_{3} \rvert + \lvert \nabla u_{4} \rvert) \lvert \nabla_{1,2} u\rvert \lvert \nabla u\rvert.
\end{equation}
Next, we estimate continuing from (30) 
\begin{align}
&\int (u\cdot\nabla) b \cdot \Delta_{1,2} b -  (b\cdot\nabla) b \cdot \Delta_{1,2} u -  (b\cdot\nabla) u \cdot \Delta_{1,2}b\\
=& -\sum_{i,j=1}^{4}\sum_{k=1}^{2}\int \partial_{k}u_{i} \partial_{i}b_{j}\partial_{k}b_{j} - \partial_{k}b_{i}\partial_{i}b_{j}\partial_{k}u_{j} - \partial_{k}b_{i}\partial_{i}u_{j}\partial_{k}b_{j}\lesssim \int \lvert \nabla b \rvert \lvert \nabla_{1,2} b\rvert \lvert \nabla u\rvert.\nonumber
\end{align}
Considering (36) and (37), we obtain (21). This completes the proof of Proposition 3.1. 
\end{proof}

With this proposition, we now obtain our first estimate:
\begin{proposition}
Let $N =4$ and $(u,b)$ be the solution pair to the MHD system (2a)-(2c) that satisfies the hypothesis of Theorem 1.3. Then $\forall \hspace{1mm} t \in (0, T], p_{i} \in [6, \infty]$,
\begin{align*}
&\sup_{\tau \in [0, t]} W(\tau) + \int_{0}^{t}Y(\tau) d\tau\\
\leq& W(0) + c\sum_{i=3}^{4}\int_{0}^{t} \lVert u_{i} \rVert_{L^{p_{i}}}^{\frac{2p_{i}}{p_{i} - 2}}X^{\frac{p_{i} - 4}{p_{i} - 2}}(\tau)Z^{\frac{2}{p_{i} - 2}}(\tau)+ \lVert b\rVert_{L^{p_{b}}}^{\frac{2p_{b}}{p_{b} - 2}}X^{\frac{p_{b} - 4}{p_{b} - 2}}(\tau)Z^{\frac{2}{p_{b} - 2}}(\tau)d\tau 
\end{align*}
with the usual convention at the case $p_{i} = \infty, i = 3 , 4, b$; i.e. $\frac{2p_{i}}{p_{i} - 2} = 2, \frac{p_{i} - 4}{p_{i} - 2} = 1, \frac{2}{p_{i} - 2} = 0$. 
\end{proposition}

\begin{proof}
We treat the case $6 \leq p_{i} < \infty \hspace{1mm} \forall \hspace{1mm} i = 3, 4, b$ first. We take $L^{2}$-inner products on (2a)-(2b) with $-\Delta_{1,2}u, -\Delta_{1,2} b$ respectively to obtain in sum 
\begin{align}
&\frac{1}{2} \partial_{t} W(t) + Y(t)\\
\lesssim& \sum_{i=3}^{4} \int \lvert u_{i} \rvert \lvert \nabla u\rvert \lvert \nabla\nabla_{1,2} u\rvert + \lvert b\rvert (\lvert \nabla u\rvert + \lvert \nabla b\rvert)(\lvert \nabla\nabla_{1,2} u\rvert + \lvert \nabla\nabla_{1,2} b\rvert) \triangleq II_{1} + II_{2} \nonumber
\end{align}
by (20). Now we estimate 
\begin{align}
II_{1} \approx& \sum_{i=3}^{4} \int \lvert u_{i} \rvert \lvert \nabla u\rvert \lvert \nabla\nabla_{1,2} u\rvert\lesssim \sum_{i=3}^{4} \lVert u_{i} \rVert_{L^{p_{i}}} \lVert \nabla u\rVert_{L^{\frac{2p_{i}}{p_{i}-2}}} \lVert \nabla\nabla_{1,2} u\rVert_{L^{2}}\\
\lesssim& \sum_{i=3}^{4}\lVert u_{i} \rVert_{L^{p_{i}}} \lVert \nabla u\rVert_{L^{2}}^{\frac{p_{i}-4}{p_{i}}} \lVert \nabla u\rVert_{L^{4}}^{\frac{4}{p_{i}}} \lVert \nabla\nabla_{1,2} u\rVert_{L^{2}}\nonumber\\
\lesssim&\sum_{i=3}^{4} \lVert u_{i} \rVert_{L^{p_{i}}} \lVert \nabla u\rVert_{L^{2}}^{\frac{p_{i}-4}{p_{i}}} \lVert \nabla\nabla_{1,2} u\rVert_{L^{2}}^{\frac{2}{p_{i}} + 1} \lVert \Delta u\rVert_{L^{2}}^{\frac{2}{p_{i}}} \nonumber\\
\leq& \frac{1}{4} \lVert \nabla\nabla_{1,2} u\rVert_{L^{2}}^{2} + c\sum_{i=3}^{4}\lVert u_{i} \rVert_{L^{p_{i}}}^{\frac{2p_{i}}{p_{i} -2}} X^{\frac{p_{i} - 4}{p_{i} - 2}}(t)Z^{\frac{2}{p_{i} -2}}(t)\nonumber
\end{align}
by H$\ddot{o}$lder's and interpolation inequalities, (18) and Young's inequalities. Similarly, 
\begin{align}
II_{2} \approx& \int \lvert b\rvert ( \lvert \nabla u\rvert + \lvert \nabla b\rvert) ( \lvert \nabla\nabla_{1,2} u \rvert + \lvert \nabla\nabla_{1,2} b\rvert)\\
\lesssim& \lVert b\rVert_{L^{p_{b}}} (\lVert \nabla u\rVert_{L^{2}}^{\frac{p_{b} - 4}{p_{b}}} + \lVert \nabla b\rVert_{L^{2}}^{\frac{p_{b} - 4}{p_{b}}})(\lVert \nabla u\rVert_{L^{4}}^{\frac{4}{p_{b}}} + \lVert \nabla b\rVert_{L^{4}}^{\frac{4}{p_{b}}}) (\lVert \nabla\nabla_{1,2} u\rVert_{L^{2}} + \lVert \nabla\nabla_{1,2} b\rVert_{L^{2}})\nonumber\\
\lesssim& \lVert b\rVert_{L^{p_{b}}} X^{\frac{p_{b} - 4}{2p_{b}}}(t)(\lVert \nabla\nabla_{1,2} u\rVert_{L^{2}}^{\frac{2}{p_{b}}} \lVert \Delta u\rVert_{L^{2}}^{\frac{2}{p_{b}}} + \lVert \nabla\nabla_{1,2} b\rVert_{L^{2}}^{\frac{2}{p_{b}}}\lVert \Delta b\rVert_{L^{2}}^{\frac{2}{p_{b}}})Y^{\frac{1}{2}}(t)\nonumber\\
\leq& \frac{1}{4} Y(t) + c\lVert b\rVert_{L^{p_{b}}}^{\frac{2p_{b}}{p_{b} - 2}}X^{\frac{p_{b} - 4}{p_{b} - 2}}(t)Z^{\frac{2}{p_{b} - 2}}(t)\nonumber
\end{align}
by H$\ddot{o}$lder's and interpolation inequalities, (19), (18) and Young's inequality. In sum of (39) and (40) in (38),  after absorbing and integrating over time $[0,t], t \in (0, T]$, we obtain the desired result in case $6 \leq p_{i} < \infty$. In case, $p_{i} = \infty$, the estimate is in fact simpler: we have 
\begin{align*}
II_{1} \lesssim& \sum_{i=3}^{4} \lVert u_{i} \rVert_{L^{\infty}} \lVert \nabla u\rVert_{L^{2}} \lVert \nabla\nabla_{1,2} u\rVert_{L^{2}}
\leq \frac{1}{4} \lVert \nabla\nabla_{1,2} u\rVert_{L^{2}}^{2} + c\sum_{i=3}^{4}\lVert u_{i} \rVert_{L^{\infty}}^{2} \lVert \nabla u\rVert_{L^{2}}^{2},\\
II_{2} \lesssim& \lVert b\rVert_{L^{\infty}} (\lVert \nabla u\rVert_{L^{2}} + \lVert \nabla b\rVert_{L^{2}})(\lVert \nabla\nabla_{1,2} u\rVert_{L^{2}} + \lVert \nabla\nabla_{1,2} b\rVert_{L^{2}}) 
\leq \frac{1}{4}Y(t) + c\lVert b\rVert_{L^{\infty}}^{2} X(t).
\end{align*}
Thus, in case $p_{i} = \infty$, Proposition 3.2 holds with $\frac{2p_{i}}{p_{i} -2} = 2, \frac{p_{i} -4}{p_{i} -2} = 1, \frac{2}{p_{i} -2} = 0$. 
\end{proof}

\subsection{$\lVert \nabla u\rVert_{L^{2}}^{2} + \lVert \nabla b\rVert_{L^{2}}^{2}$-estimate}
The next important step of the proof is to make use of the $\lVert \nabla_{1,2} u\rVert_{L^{2}}^{2} + \lVert \nabla_{1,2} b\rVert_{L^{2}}^{2}$-estimate to obtain the bound on $\lVert \nabla u\rVert_{L^{2}}^{2} + \lVert \nabla  b\rVert_{L^{2}}^{2}$, which requires another key decomposition (see (43), (46)). 
\begin{proposition}
Let $N =4$ and $(u,b)$ be the solution pair to the MHD system (2a)-(2c) that satisfies the hypothesis of Theorem 1.3. Then 
\begin{align*}
\sup_{t \in [0,T]} X(t) + \int_{0}^{T} Z(\tau) d\tau \lesssim 1. 
\end{align*}
\end{proposition}

\begin{proof}
Firstly, we assume $6 \leq p_{i} < \infty$ again. We take $L^{2}$-inner products on (2a)-(2b) with $(-\Delta u, -\Delta b)$ respectively to obtain 
\begin{align}
&\frac{1}{2} \partial_{t} X(t) + Z(t)\\
=& \int (u\cdot\nabla) u \cdot \Delta_{1,2} u + (u\cdot\nabla) u \cdot \Delta_{3,4} u + (u\cdot\nabla) b \cdot \Delta_{1,2} b +  (u\cdot\nabla) b \cdot \Delta_{3,4} b \nonumber\\
& -  (b\cdot\nabla) b \cdot \Delta_{1,2} u - (b\cdot\nabla) b \cdot \Delta_{3,4} u -(b\cdot\nabla) u \cdot \Delta_{1,2} b - (b\cdot\nabla) u \cdot \Delta_{3,4} b\triangleq \sum_{i=1}^{8} III_{i}.\nonumber
\end{align}
From (38)-(40), we already have the estimates of 
\begin{align}
&III_{1} + III_{3} + III_{5} + III_{7}\lesssim II_{1} + II_{2}\\
\lesssim& \sum_{i=3}^{4}\lVert u_{i} \rVert_{L^{p_{i}}} \lVert \nabla u\rVert_{L^{2}}^{\frac{p_{i} -4}{p_{i}}} \lVert \nabla \nabla_{1,2} u\rVert_{L^{2}}^{\frac{2+p_{i}}{p_{i}}} \lVert \Delta u\rVert_{L^{2}}^{\frac{2}{p_{i}}} + \lVert b\rVert_{L^{p_{b}}} X^{\frac{p_{b} - 4}{2p_{b}}}(t)Z^{\frac{1}{p_{b}}}(t)Y^{\frac{1}{p_{b}} + \frac{1}{2}}(t)\nonumber\\
\leq& \frac{1}{16} Z(t) + c\sum_{i=3}^{4}(\lVert u_{i} \rVert_{L^{p_{i}}}^{\frac{2p_{i}}{p_{i} - 4}} + \lVert b\rVert_{L^{p_{b}}}^{\frac{2p_{b}}{p_{b} - 4}})X(t)\nonumber
\end{align}
by Young's inequalities. Next, we work on $III_{2}$, which we first integrate by parts and decompose as follows: 
\begin{align}
III_{2} =&\int (u\cdot\nabla) u \cdot \Delta_{3,4} u= -\sum_{i,j=1}^{4}\sum_{k=3}^{4}\int \partial_{k}u_{i}\partial_{i}u_{j}\partial_{k}u_{j}\\
=& -\sum_{i=1}^{2}\sum_{j=1}^{4}\sum_{k=3}^{4}\int \partial_{k}u_{i}\partial_{i}u_{j}\partial_{k}u_{j} - \sum_{j=1}^{4}\sum_{i,k=3}^{4}\int\partial_{k}u_{i}\partial_{i}u_{j}\partial_{k}u_{j}\nonumber\\
=& -\sum_{i=1}^{2}\sum_{j=1}^{4}\sum_{k=3}^{4}\int\partial_{k}u_{i}\partial_{i}u_{j}\partial_{k}u_{j} + \sum_{j=1}^{4}\sum_{i,k=3}^{4}\int u_{i}\partial_{k}(\partial_{i}u_{j}\partial_{k}u_{j})\nonumber\\
\lesssim& \int \lvert \nabla u\rvert^{2} \lvert \nabla_{1,2} u\rvert + \sum_{i=3}^{4} \int \lvert u_{i} \rvert \lvert \nabla u\rvert \lvert \nabla^{2} u\rvert \triangleq IV_{1} + IV_{2}.\nonumber
\end{align}
We estimate 
\begin{align}
IV_{1} \approx& \int \lvert \nabla_{1,2} u\rvert \lvert \nabla u\rvert^{2} \lesssim \lVert \nabla_{1,2} u\rVert_{L^{2}} \lVert \nabla u\rVert_{L^{4}}^{2}\\
\lesssim& \lVert \nabla_{1,2} u\rVert_{L^{2}} \lVert \nabla\nabla_{1,2} u\rVert_{L^{2}} \lVert \Delta u\rVert_{L^{2}}\lesssim W^{\frac{1}{2}}(t)Y^{\frac{1}{2}}(t) Z^{\frac{1}{2}}(t)\nonumber
\end{align}
by H$\ddot{o}$lder's inequalities and (18). On the other hand, 
\begin{align}
IV_{2} \lesssim& \sum_{i=3}^{4}\lVert u_{i} \rVert_{L^{p_{i}}} \lVert \nabla u\rVert_{L^{\frac{2p_{i}}{p_{i} -2}}} \lVert \nabla^{2} u\rVert_{L^{2}}\\
\lesssim& \sum_{i=3}^{4}\lVert u_{i}\rVert_{L^{p_{i}}} \lVert \nabla u\rVert_{L^{2}}^{1-\frac{4}{p_{i}}} \lVert \Delta u\rVert_{L^{2}}^{1+\frac{4}{p_{i}}} \leq \frac{1}{16}Z(t) + c\sum_{i=3}^{4}\lVert u_{i} \rVert_{L^{p_{i}}}^{\frac{2p_{i}}{p_{i} -4}} X(t)  \nonumber
\end{align}
by H$\ddot{o}$lder's, Gagliardo-Nirenberg and Young's inequalities. Next, again we carefully decompose
\begin{align}
III_{4} =& \int (u\cdot\nabla) b \cdot \Delta_{3,4} b = -\sum_{i,j=1}^{4}\sum_{k=3}^{4}\int \partial_{k}u_{i}\partial_{i}b_{j}\partial_{k}b_{j} \\
=& -\sum_{i=1}^{2}\sum_{j=1}^{4}\sum_{k=3}^{4}\int \partial_{k}u_{i}\partial_{i}b_{j}\partial_{k}b_{j} - \sum_{j=1}^{4}\sum_{i,k=3}^{4}\int \partial_{k}u_{i}\partial_{i}b_{j}\partial_{k}b_{j}\nonumber\\
=& -\sum_{i=1}^{2}\sum_{j=1}^{4}\sum_{k=3}^{4}\int \partial_{k}u_{i}\partial_{i}b_{j}\partial_{k}b_{j}  + \sum_{j=1}^{4}\sum_{i,k=3}^{4}\int u_{i}\partial_{k} (\partial_{i}b_{j}\partial_{k}b_{j})\nonumber\\
\lesssim&\int \lvert \nabla u\rvert \lvert \nabla_{1,2} b\rvert \lvert \nabla b\rvert + \sum_{i=3}^{4} \int \lvert u_{i}\rvert \lvert \nabla b\rvert \lvert \nabla^{2} b\rvert \triangleq IV_{3} + IV_{4}.\nonumber
\end{align}
We estimate 
\begin{align}
IV_{3} \approx& \int \lvert \nabla u\rvert \lvert \nabla_{1,2} b\rvert \lvert \nabla b\rvert \lesssim \lVert \nabla_{1,2} b\rVert_{L^{2}} \lVert \nabla u\rVert_{L^{4}} \lVert \nabla b\rVert_{L^{4}}\\
\lesssim& \lVert \nabla_{1,2} b\rVert_{L^{2}} \lVert \nabla\nabla_{1,2} u\rVert_{L^{2}}^{\frac{1}{2}} \lVert \nabla\nabla_{1,2} b\rVert_{L^{2}}^{\frac{1}{2}} \lVert \Delta u\rVert_{L^{2}}^{\frac{1}{2}} \lVert \Delta b\rVert_{L^{2}}^{\frac{1}{2}}\lesssim W^{\frac{1}{2}}(t) Y^{\frac{1}{2}}(t) Z^{\frac{1}{2}}(t)\nonumber
\end{align}
by H$\ddot{o}$lder's inequalities, (18) and Young's inequalities. On the other hand, we estimate similarly to $IV_{2}$ in (45), 
\begin{align}
IV_{4} \lesssim& \sum_{i=3}^{4} \lVert u_{i} \rVert_{L^{p_{i}}} \lVert \nabla b\rVert_{L^{\frac{2p_{i}}{p_{i} -2}}} \lVert \nabla^{2} b\rVert_{L^{2}}\\
\lesssim& \sum_{i=3}^{4}  \lVert u_{i}\rVert_{L^{p_{i}}} \lVert \nabla b\rVert_{L^{2}}^{1-\frac{4}{p_{i}}} \lVert \Delta b\rVert_{L^{2}}^{1+\frac{4}{p_{i}}} \leq \frac{1}{16} Z(t) + c\sum_{i=3}^{4} \lVert u_{i} \rVert_{L^{p_{i}}}^{\frac{2p_{i}}{p_{i} -4}} X(t)  \nonumber
\end{align}
by H$\ddot{o}$lder's, Gagliardo-Nirenberg and Young's inequalities. Finally, similarly to $IV_{2}$ in (45) again 
\begin{align}
III_{6} + III_{8} =& -\int (b\cdot\nabla) b \cdot \Delta_{3,4} u + (b\cdot\nabla) u \cdot \Delta_{3,4} b\\
\lesssim& \lVert b\rVert_{L^{p_{b}}} (\lVert \nabla b\rVert_{L^{2}}^{1-\frac{4}{p_{b}}} + \lVert \nabla u\rVert_{L^{2}}^{1-\frac{4}{p_{b}}}) (\lVert \Delta u\rVert_{L^{2}}^{1+ \frac{4}{p_{b}}} + \lVert \Delta b\rVert_{L^{2}}^{1+ \frac{4}{p_{b}}})\nonumber\\
\leq& \frac{1}{16} Z(t) + c\lVert b\rVert_{L^{p_{b}}}^{\frac{2p_{b}}{p_{b} - 4}}X(t)\nonumber
\end{align}
by H$\ddot{o}$lder's, Gagliardo-Nirenberg and Young's inequalities. Thus, applying (42)-(49) in (41), we obtain after absorbing 
\begin{align}
\frac{1}{2} \partial_{t} X + \frac{1}{2}Z(t)
\lesssim \sum_{i=3}^{4}(\lVert u_{i} \rVert_{L^{p_{i}}}^{\frac{2p_{i}}{p_{i} - 4}} + \lVert b\rVert_{L^{p_{b}}}^{\frac{2p_{b}}{p_{b} - 4}})X(t) + W^{\frac{1}{2}}(t)Y^{\frac{1}{2}}(t)Z^{\frac{1}{2}}(t).
\end{align}
Now we assume $6 < p_{i} < \infty$. Integrating over $[0,t], t \in (0, T]$, we obtain 
\begin{align*}
&X(t) + \int_{0}^{t} Z(\tau) d\tau\\
\leq& X(0) + c\sum_{i=3}^{4}\int_{0}^{t} (\lVert u_{i} \rVert_{L^{p_{i}}}^{\frac{2p_{i}}{p_{i} - 4}}  + \lVert b\rVert_{L^{p_{b}}}^{\frac{2p_{b}}{p_{b} - 4}})X(\tau) d\tau + c\int_{0}^{t} W^{\frac{1}{2}}(\tau) Y^{\frac{1}{2}}(\tau) Z^{\frac{1}{2}}(\tau) d\tau.
\end{align*}
We focus only on the last integral which we bound by a constant multiples of 
\begin{align*}
&\sup_{\tau \in [0,t]} W^{\frac{1}{2}}(\tau) \left(\int_{0}^{t} Y(\tau) d\tau\right)^{\frac{1}{2}} \left(\int_{0}^{t} Z(\tau) d\tau\right)^{\frac{1}{2}}\\
\lesssim& \left(W(0) + \sum_{i=3}^{4}\int_{0}^{t} \lVert u_{i} \rVert_{L^{p_{i}}}^{\frac{2p_{i}}{p_{i} - 2}}X^{\frac{p_{i} - 4}{p_{i} - 2}}(\tau)Z^{\frac{2}{p_{i} - 2}}(\tau) + \lVert b\rVert_{L^{p_{b}}}^{\frac{2p_{b}}{p_{b} - 2}}X^{\frac{p_{b} - 4}{p_{b} - 2}}(\tau)Z^{\frac{2}{p_{b} - 2}}(\tau)d\tau  \right)\\
&\times \left(\int_{0}^{t} Z(\tau) d\tau\right)^{\frac{1}{2}}\\
\lesssim& \left(\int_{0}^{t} Z(\tau)d\tau\right)^{\frac{1}{2}} +\sum_{i=3}^{4}\left(\int_{0}^{t} \lVert u_{i} \rVert_{L^{p_{i}}}^{\frac{2p_{i}}{p_{i} -4}} X(\tau) d\tau\right)^{\frac{p_{i} -4}{p_{i} -2}} \left(\int_{0}^{t} Z(\tau) d\tau\right)^{\frac{p_{i} + 2}{2(p_{i} -2)}}\\
&+ \left(\int_{0}^{t} \lVert b \rVert_{L^{p_{b}}}^{\frac{2p_{b}}{p_{b} -4}} X(\tau) d\tau\right)^{\frac{p_{b} -4}{p_{b} -2}} \left(\int_{0}^{t} Z(\tau) d\tau\right)^{\frac{p_{b} + 2}{2(p_{b} -2)}}\\
\leq& \frac{1}{2}\int_{0}^{t} Z(\tau) d\tau + c(1 + \sum_{i=3}^{4}\left(\int_{0}^{t} \lVert u_{i} \rVert_{L^{p_{i}}}^{\frac{4p_{i}}{p_{i} - 6}}X(\tau) d\tau\right) \left(\int_{0}^{t} X(\tau) d\tau\right)^{\frac{p_{i} - 2}{p_{i} - 6}}\\
& \hspace{25mm}  + \left(\int_{0}^{t} \lVert b \rVert_{L^{p_{b}}}^{\frac{4p_{b}}{p_{b} - 6}}X(\tau) d\tau\right) \left(\int_{0}^{t} X(\tau) d\tau\right)^{\frac{p_{b} - 2}{p_{b} - 6}})\\
\leq& \frac{1}{2} \int_{0}^{t} Z(\tau) d\tau + c\left(1+  \sum_{i=3}^{4}\int_{0}^{t} (\lVert u_{i}\rVert_{L^{p_{i}}}^{\frac{4p_{i}}{p_{i} -6}} + \lVert b\rVert_{L^{p_{b}}}^{\frac{4p_{b}}{p_{b} -6}}) X(\tau) d\tau \right)
\end{align*}
by H$\ddot{o}$lder's inequalities, Proposition 3.2, Young's inequalities and (3). After absorbing, Gronwall's inequality implies the desired result in case $6 < p_{i} < \infty, r_{i} < \infty$. 

We now consider the case $p_{i} = \infty$, assuming for the simplicity of presentation that $p_{3} = p_{4} = p_{b} = \infty$. Firstly, we could have computed in contrast to (42), (43), (46) and (49) respectively 
\begin{align}
&III_{1} + III_{3} + III_{5} + III_{7}\lesssim II_{1} + II_{2}\\
\lesssim& \sum_{i=3}^{4} \lVert u_{i} \rVert_{L^{\infty}} \lVert \nabla u\rVert_{L^{2}} \lVert \Delta u\rVert_{L^{2}} + \lVert b\rVert_{L^{\infty}}(\lVert \nabla u\rVert_{L^{2}} + \lVert \nabla b\rVert_{L^{2}}) (\lVert \Delta u\rVert_{L^{2}} + \lVert \Delta b\rVert_{L^{2}})\nonumber\\
\leq& \frac{1}{16} Z(t) + c\sum_{i=3}^{4}(\lVert u_{i} \rVert_{L^{\infty}}^{2} + \lVert b\rVert_{L^{\infty}}^{2})X(t),\nonumber
\end{align}
\begin{align}
III_{2} \leq& IV_{1} + IV_{2} \\
\lesssim& \lVert \nabla_{1,2} u\rVert_{L^{2}} \lVert \nabla u\rVert_{L^{4}}^{2} + \sum_{i=3}^{4}\lVert u_{i} \rVert_{L^{\infty}}\lVert \nabla u\rVert_{L^{2}} \lVert \nabla^{2} u\rVert_{L^{2}}\nonumber\\
\leq& \frac{1}{16} Z(t) + c\left(W^{\frac{1}{2}}(t)Y^{\frac{1}{2}}(t)Z^{\frac{1}{2}}(t) + \sum_{i=3}^{4}\lVert u_{i} \rVert_{L^{\infty}}^{2}X(t)\right),\nonumber
\end{align}
\begin{align}
III_{4} \lesssim& IV_{3} + IV_{4}\\
\lesssim& \lVert \nabla u\rVert_{L^{4}} \lVert \nabla_{1,2} b\rVert_{L^{2}} \lVert \nabla b\rVert_{L^{4}} + \sum_{i=3}^{4}\lVert u_{i} \rVert_{L^{\infty}}\lVert \nabla b\rVert_{L^{2}} \lVert \Delta b\rVert_{L^{2}}\nonumber\\
\lesssim& \lVert \nabla_{1,2} b\rVert_{L^{2}} \lVert \nabla\nabla_{1,2} u\rVert_{L^{2}}^{\frac{1}{2}} \lVert \Delta u\rVert_{L^{2}}^{\frac{1}{2}} \lVert \nabla\nabla_{1,2} b\rVert_{L^{2}}^{\frac{1}{2}} \lVert \Delta b\rVert_{L^{2}}^{\frac{1}{2}} + \sum_{i=3}^{4}\lVert u_{i} \rVert_{L^{\infty}}\lVert \nabla b\rVert_{L^{2}} \lVert \Delta b\rVert_{L^{2}}\nonumber\\
\leq& \frac{1}{16} Z(t) + c\left(W^{\frac{1}{2}}(t)Y^{\frac{1}{2}}(t) Z^{\frac{1}{2}}(t) + \sum_{i=3}^{4}\lVert u_{i} \rVert_{L^{\infty}}^{2} X(t)\right),\nonumber
\end{align}
\begin{align}
III_{6} + III_{8} 
\lesssim& \lVert b\rVert_{L^{\infty}} (\lVert \nabla b\rVert_{L^{2}} + \lVert \nabla u\rVert_{L^{2}})(\lVert \nabla^{2} u\rVert_{L^{2}} + \lVert \nabla^{2} b\rVert_{L^{2}})\\
\leq& \frac{1}{16} Z(t) + c\lVert b\rVert_{L^{\infty}}^{2} X(t)\nonumber
\end{align}
all by H$\ddot{o}$lder's and Young's inequalities and (18) only in (52) and (53). Thus applying (51)-(54) in (41), absorbing and integrating in time $[0, t]$, we obtain 
\begin{align*}
&X(t) + \frac{3}{2}\int_{0}^{t}Z(\tau) d\tau\\
\leq& X(0) + c\sum_{i=3}^{4}\int_{0}^{t} (\lVert u_{i} \rVert_{L^{\infty}}^{2} + \lVert b\rVert_{L^{\infty}}^{2})X(\tau) d\tau\\
& \hspace{8mm} + c\sup_{\tau \in [0,t]} W^{\frac{1}{2}}(\tau) \left(\int_{0}^{t} Y(\tau) d\tau\right)^{\frac{1}{2}} \left(\int_{0}^{t} Z(\tau) d\tau \right)^{\frac{1}{2}}\\
\leq& \frac{1}{2} \int_{0}^{t}Z(\tau) d\tau + c \sum_{i=3}^{4}\int_{0}^{t} (\lVert u_{i} \rVert_{L^{\infty}}^{2}+ \lVert b\rVert_{L^{\infty}}^{2})X(\tau) d\tau\\
+&c\left(W(0) + \sum_{i=3}^{4}\int_{0}^{t} (\lVert u_{i} \rVert_{L^{\infty}}^{2} + \lVert b\rVert_{L^{\infty}}^{2}) X(\tau) d\tau)\right)^{2}\\
\leq&  \frac{1}{2} \int_{0}^{t}Z(\tau) d\tau\\
&+c\left( \sum_{i=3}^{4}\int_{0}^{t} (\lVert u_{i} \rVert_{L^{\infty}}^{2} + \lVert b\rVert_{L^{\infty}}^{2})X(\tau) d\tau + 1+ \sum_{i=3}^{4}\int_{0}^{t} (\lVert u_{i} \rVert_{L^{\infty}}^{4} + \lVert b\rVert_{L^{\infty}}^{4}) X(\tau) d\tau \right)
\end{align*}
by H$\ddot{o}$lder's inequality, Proposition 3.2, Young's inequality, (19) and (3). This completes the proof in case $p_{i} = \infty$. 

We now prove the second statement of Theorem 1.3, namely the smallness result when $p_{i} = 6, r_{i} = \infty$. For simplicity of presentation, we assume $p_{i} = 6 \hspace{1mm} \forall \hspace{1mm} i = 3, 4, b$. We integrate in time on (50) to obtain  
\begin{align*}
&X(t) + \int_{0}^{t} Z(\tau) d\tau\\
\leq& X(0) + c(\sum_{i=3}^{4}\sup_{\tau \in [0,t]} (\lVert u_{i} \rVert_{L^{6}}^{6} + \lVert b\rVert_{L^{6}}^{6})(\tau) \int_{0}^{t} X(\tau)d\tau + \sup_{\tau \in [0,t]} W^{\frac{1}{2}}(t) \left(\int_{0}^{t} Y(\tau) d\tau\right)^{\frac{1}{2}}\\
& \hspace{10mm} \times\left(\int_{0}^{t}Z(\tau) d\tau\right)^{\frac{1}{2}})\\
\lesssim& 1 + \left(W(0) + \sum_{i=3}^{4}\int_{0}^{t} (\lVert u_{i} \rVert_{L^{6}}^{3} + \lVert b\rVert_{L^{6}}^{3}) X^{\frac{1}{2}}(\tau) Z^{\frac{1}{2}}(\tau) d\tau \right) \left(\int_{0}^{t} Z(\tau) d\tau \right)^{\frac{1}{2}}\\
\lesssim& 1 + \left(\int_{0}^{t} Z(\tau) d\tau \right)^{\frac{1}{2}} + \sum_{i=3}^{4}\sup_{\tau \in [0,t]} (\lVert u_{i} \rVert_{L^{6}}^{3} + \lVert b\rVert_{L^{6}}^{3})(\tau) \left(\int_{0}^{t} X(\tau) d\tau\right)^{\frac{1}{2}} \left(\int_{0}^{t} Z(\tau) d\tau\right)\\
\leq& \frac{1}{2} \int_{0}^{t}Z(\tau) d\tau + c
\end{align*}
for $\sum_{i=3}^{4}\sup_{\tau \in [0,t]} (\lVert u_{i} \rVert_{L^{6}}^{3} + \lVert b\rVert_{L^{6}}^{3})(\tau)$ sufficiently small where we used H$\ddot{o}$lder's inequality, Proposition 3.2, Young's inequality and (3). Absorbing, Gronwall's inequality completes the proof of Theorem 1.3. 
\end{proof}

\section{Proof of Theorem 1.4}
We assume for simplicity of presentation that $\forall \hspace{1mm} i = 3, 4, b, p_{i} \in [\frac{12}{5}, 4]$ or $p_{i} \in [4, \infty]$. A combination of mixed cases can be obtained following the proofs below. 
\begin{proposition}
Let $N =4$ and $(u,b)$ be the solution pair to the MHD system (2a)-(2c) that satisfies the hypothesis of Theorem 1.4. Then $\forall \hspace{1mm} t \in (0, T]$,
\begin{align*}
&\sup_{\tau \in [0, t]} W(\tau) + \int_{0}^{t}Y(\tau) d\tau\\
\leq&
\begin{cases}
W(0) + c\sum_{i=3}^{4} \int_{0}^{t} \lVert \nabla u_{i} \rVert_{L^{p_{i}}}^{\frac{4p_{i}}{3p_{i} -4}} X^{\frac{4(p_{i} -2)}{3p_{i} -4}}(\tau) Z^{\frac{4-p_{i}}{3p_{i} -4}}(\tau)\\
 \hspace{30mm} + \lVert \nabla b\rVert_{L^{p_{b}}}^{\frac{4p_{b}}{3p_{b} -4}} X^{\frac{4(p_{b} -2)}{3p_{b} - 4}} (\tau)Z^{\frac{4-p_{b}}{3p_{b}  -4}}(\tau) d\tau, & \text{ if } p_{i} \in [\frac{12}{5}, 4],\\
W(0) + c\sum_{i=3}^{4} \int_{0}^{t} (\lVert \nabla u_{i} \rVert_{L^{p_{i}}}^{\frac{p_{i}}{p_{i} -2}} + \lVert \nabla b\rVert_{L^{p_{b}}}^{\frac{p_{b}}{p_{b} -2}})X(\tau) d\tau, & \text{ if } p_{i} \in [4, \infty],
 \end{cases}
\end{align*}
with the usual convention at $p_{i} = \infty, i = 3, 4, b$; i.e. $\frac{p_{i}}{p_{i} - 2} = 1$. 
\end{proposition}

\begin{proof}
We first assume $p_{i} \in \left[\frac{12}{5}, 4\right]$. We take $L^{2}$-inner products of (2a)-(2b) with $-\Delta_{1,2} u, -\Delta_{1,2} b$ respectively and estimate 
\begin{align}
\frac{1}{2}\partial_{t} W(t) + Y(t) 
\lesssim \sum_{i=3}^{4} \int \lvert \nabla u_{i} \rvert\lvert \nabla_{1,2}u \rvert \lvert \nabla u\rvert + \lvert \nabla b\rvert \lvert \nabla_{1,2} b\rvert \lvert \nabla u\rvert
\end{align}
by (21). Now we estimate 
\begin{align}
\sum_{i=3}^{4}\int \lvert \nabla u_{i} \rvert\lvert \nabla_{1,2} u\rvert \lvert \nabla u\rvert \lesssim& \sum_{i=3}^{4} \lVert \nabla u_{i} \rVert_{L^{p_{i}}} \lVert \nabla_{1,2} u\rVert_{L^{4}} \lVert \nabla u\rVert_{L^{\frac{4p_{i}}{3p_{i} -4}}} \\
\lesssim& \sum_{i=3}^{4} \lVert \nabla u_{i} \rVert_{L^{p_{i}}} \lVert \nabla\nabla_{1,2} u\rVert_{L^{2}} \lVert \nabla u\rVert_{L^{2}}^{2(\frac{p_{i} -2}{p_{i}})} \lVert \nabla u\rVert_{L^{4}}^{\frac{4-p_{i}}{p_{i}}}\nonumber\\
\lesssim& \sum_{i=3}^{4} \lVert \nabla u_{i} \rVert_{L^{p_{i}}} \lVert \nabla\nabla_{1,2} u\rVert_{L^{2}}^{\frac{4+p_{i}}{2p_{i}}} \lVert \nabla u\rVert_{L^{2}}^{2(\frac{p_{i} -2}{p_{i}})}  \lVert \Delta u\rVert_{L^{2}}^{\frac{4-p_{i}}{2p_{i}}}\nonumber\\
\leq& \frac{1}{4} Y(t) + c\sum_{i=3}^{4} \lVert \nabla u_{i} \rVert_{L^{p_{i}}}^{\frac{4p_{i}}{3p_{i} - 4}}X^{\frac{4(p_{i} - 2)}{3p_{i} - 4}}(t)Z^{\frac{4-p_{i}}{3p_{i} - 4}}(t)\nonumber
\end{align}
by H$\ddot{o}$lder's inequalities, Sobolev embedding of $\dot{H}^{1}(\mathbb{R}^{4}) \hookrightarrow L^{4}(\mathbb{R}^{4})$, interpolation inequality, (18) and Young's inequality. Similarly, we obtain 
\begin{align}
\int \lvert \nabla b \rvert \lvert \nabla_{1,2} b\rvert \lvert \nabla u\rvert
\lesssim& \lVert \nabla b\rVert_{L^{p_{b}}}\lVert \nabla\nabla_{1,2} b\rVert_{L^{2}} \lVert \nabla u\rVert_{L^{\frac{4p_{b}}{3p_{b} - 4}}}\\
\lesssim& \lVert \nabla b\rVert_{L^{p_{b}}} Y^{\frac{1}{2}}(t) \lVert \nabla u\rVert_{L^{2}}^{2(\frac{p_{b} - 2}{p_{b}})}\lVert \nabla u\rVert_{L^{4}}^{\frac{4-p_{b}}{p_{b}}}\nonumber\\
\lesssim& \lVert \nabla b\rVert_{L^{p_{b}}}Y^{\frac{1}{2}}(t) X^{\frac{p_{b} -2}{p_{b}}} (t) \lVert \nabla\nabla_{1,2} u\rVert_{L^{2}}^{\frac{4-p_{b}}{2p_{b}}}\lVert \Delta u\rVert_{L^{2}}^{\frac{4-p_{b}}{2p_{b}}}\nonumber\\
\leq& \frac{1}{4} Y(t) + c\lVert \nabla b\rVert_{L^{p_{b}}}^{\frac{4p_{b}}{3p_{b} - 4}}X^{\frac{4(p_{b}  -2)}{3p_{b} - 4}}(t) Z^{\frac{4-p_{b}}{3p_{b} - 4}}(t).\nonumber
\end{align}
With (56) and (57) applied to (55), absorbing and integrating in time lead to 
\begin{align}
&W(t) + \int_{0}^{t} Y(\tau) d\tau\\
\leq& W(0)\nonumber\\
&+c\sum_{i=3}^{4} \int_{0}^{t} \lVert \nabla u_{i} \rVert_{L^{p_{i}}}^{\frac{4p_{i}}{3p_{i} -4}}X^{\frac{4(p_{i} -2)}{3p_{i} -4}}(\tau) Z^{\frac{4-p_{i}}{3p_{i} -4}}(\tau) + \lVert \nabla b\rVert_{L^{p_{b}}}^{\frac{4p_{b}}{3p_{b} -4}} X^{\frac{4(p_{b} -2)}{3p_{b} - 4}}(\tau) Z^{\frac{4-p_{b}}{3p_{b}  -4}}(\tau) d\tau.\nonumber
\end{align}
We now work on the case $4 < p_{i} < \infty$: 
\begin{align}
\sum_{i=3}^{4}\int \lvert \nabla u_{i} \rvert \lvert \nabla_{1,2} u\rvert \lvert \nabla u\rvert \lesssim& \sum_{i=3}^{4} \lVert \nabla u_{i} \rVert_{L^{p_{i}}} \lVert \nabla_{1,2} u\rVert_{L^{\frac{2p_{i}}{p_{i} - 2}}}\lVert \nabla u\rVert_{L^{2}}\\
\lesssim& \sum_{i=3}^{4} \lVert \nabla u_{i} \rVert_{L^{p_{i}}} \lVert \nabla_{1,2} u\rVert_{L^{2}}^{\frac{p_{i} - 4}{p_{i}}}\lVert \nabla_{1,2} u\rVert_{L^{4}}^{\frac{4}{p_{i}}}\lVert \nabla u\rVert_{L^{2}}\nonumber\\
\lesssim& \sum_{i=3}^{4} \lVert \nabla u_{i} \rVert_{L^{p_{i}}} \lVert \nabla_{1,2} u\rVert_{L^{2}}^{\frac{p_{i} - 4}{p_{i}}}\lVert \nabla\nabla_{1,2} u\rVert_{L^{2}}^{\frac{4}{p_{i}}}\lVert \nabla u\rVert_{L^{2}}\nonumber\\
\leq& \frac{1}{4}Y(t) + c\sum_{i=3}^{4} \lVert \nabla u_{i}\rVert_{L^{p_{i}}}^{\frac{p_{i}}{p_{i} - 2}}X(t)\nonumber
\end{align}
by H$\ddot{o}$lder's and interpolation inequalities, Sobolev embedding of $\dot{H}^{1}(\mathbb{R}^{4})\hookrightarrow L^{4}(\mathbb{R}^{4})$ and Young's inequality. Similarly, we estimate 
\begin{align}
\int \lvert \nabla b\rvert \lvert \nabla_{1,2} b\rvert \lvert \nabla u\rvert
\lesssim& \lVert \nabla b\rVert_{L^{p_{b}}} \lVert \nabla_{1,2} b\rVert_{L^{2}}^{\frac{p_{b} -4}{p_{b}}} \lVert \nabla_{1,2} b\rVert_{L^{4}}^{\frac{4}{p_{b}}} \lVert \nabla u\rVert_{L^{2}}\\
\lesssim& \lVert \nabla b\rVert_{L^{p_{b}}} \lVert \nabla_{1,2} b\rVert_{L^{2}}^{\frac{p_{b} -4}{p_{b}}} \lVert\nabla \nabla_{1,2} b\rVert_{L^{2}}^{\frac{4}{p_{b}}} \lVert \nabla u\rVert_{L^{2}}\nonumber\\
\leq& \frac{1}{4} Y(t) + c\lVert \nabla b\rVert_{L^{p_{b}}}^{\frac{p_{b}}{p_{b} - 2}}X(t).\nonumber
\end{align}
We apply (59) and (60) in (55), absorb and integrate in time to obtain 
\begin{align}
W(t) + \int_{0}^{t} Y(\tau) d\tau
\leq W(0) + c\sum_{i=3}^{4} \int_{0}^{t} (\lVert \nabla u_{i} \rVert_{L^{p_{i}}}^{\frac{p_{i}}{p_{i} -2}} + \lVert \nabla b\rVert_{L^{p_{b}}}^{\frac{p_{b}}{p_{b} -2}}) X(\tau) d\tau.
\end{align}
The case $p_{i} = \infty$ requires only a standard modification as done in the proof of Theorem 1.3; that is, 
\begin{align*}
\sum_{i=3}^{4}\int \lvert \nabla u_{i} \rvert \lvert \nabla_{1,2} u\rvert \lvert \nabla u\rvert
\lesssim \sum_{i=3}^{4} \lVert \nabla u_{i} \rVert_{L^{\infty}} \lVert \nabla_{1,2} u\rVert_{L^{2}} \lVert \nabla u\rVert_{L^{2}} \lesssim \sum_{i=3}^{4} \lVert \nabla u_{i} \rVert_{L^{\infty}} X(t),
\end{align*}
\begin{align*}
\int \lvert \nabla b\rvert \lvert \nabla_{1,2} b\rvert \lvert \nabla u\rvert
\lesssim \lVert \nabla b\rVert_{L^{\infty}} \lVert \nabla_{1,2} b\rVert_{L^{2}} \lVert \nabla u\rVert_{L^{2}} \lesssim \lVert \nabla b\rVert_{L^{\infty}}X(t)
\end{align*}
so that summing and integrating in time leads to the desired result. This completes the proof of Proposition 4.1.
\end{proof}
\begin{proposition}
Let $N =4$ and $(u,b)$ be the solution pair to the MHD system (2a)-(2c) that satisfies the hypothesis of Theorem 1.4. Then 
\begin{align*}
\sup_{t \in [0,T]} X(t) + \int_{0}^{T} Z(\tau) d\tau \lesssim 1. 
\end{align*}
\end{proposition}

\begin{proof}
Similarly to the proof of Theorem 1.3, we estimate from (41). For $p_{i} \in [\frac{12}{5}, 4]$, we continue our estimate from (55), (56) and (57) to obtain 
\begin{align}
&III_{1} +III_{3} + III_{5} + III_{7}\\
\lesssim& \sum_{i=3}^{4} \lVert \nabla u_{i} \rVert_{L^{p_{i}}} \lVert \nabla\nabla_{1,2} u\rVert_{L^{2}}^{\frac{4+p_{i}}{2p_{i}}}\lVert \nabla u\rVert_{L^{2}}^{2(\frac{p_{i} - 2}{p_{i}})}\lVert \Delta u\rVert_{L^{2}}^{\frac{4-p_{i}}{2p_{i}}}\nonumber\\
&+ \lVert \nabla b\rVert_{L^{p_{b}}} Y^{\frac{4+p_{b}}{4p_{b}}}(t) X^{\frac{p_{b} - 2}{p_{b}}}(t)\lVert \Delta u\rVert_{L^{2}}^{\frac{4-p_{b}}{2p_{b}}}\nonumber\\
\leq& \frac{1}{16} Z(t) + c\sum_{i=3}^{4} \left(\lVert \nabla u_{i} \rVert_{L^{p_{i}}}^{\frac{p_{i}}{p_{i} - 2}} + \lVert \nabla b\rVert_{L^{p_{b}}}^{\frac{p_{b}}{p_{b} -2}} \right)X(t)\nonumber
\end{align}
by Young's inequality. We now decompose integrating by parts 
\begin{align}
III_{2}
=& -\sum_{i,j=1}^{4}\sum_{k=3}^{4}\int \partial_{k}u_{i}\partial_{i}u_{j} \partial_{k}u_{j}\\
=& -\sum_{i=1}^{2}\sum_{j=1}^{4}\sum_{k=3}^{4}\int \partial_{k}u_{i}\partial_{i}u_{j}\partial_{k}u_{j} - \sum_{j=1}^{4}\sum_{i,k=3}^{4}\int \partial_{k}u_{i}\partial_{i}u_{j}\partial_{k}u_{j}\nonumber\\
\lesssim& \int \lvert \nabla u\rvert^{2} \lvert \nabla_{1,2} u\rvert + \sum_{i=3}^{4} \int \lvert \nabla u_{i} \rvert \lvert \nabla u\rvert^{2} \triangleq V_{1} + V_{2}\nonumber
\end{align}
where $V_{1}$ is estimated identically as $IV_{1}$ in (44) while we estimate  
\begin{align}
V_{2} \lesssim& \sum_{i=3}^{4} \lVert \nabla u_{i} \rVert_{L^{p_{i}}} \lVert \nabla u\rVert_{L^{\frac{2p_{i}}{p_{i} - 1}}}^{2} \\
\lesssim& \sum_{i=3}^{4} \lVert \nabla u_{i} \rVert_{L^{p_{i}}} \lVert \nabla u\rVert_{L^{2}}^{2(\frac{p_{i} - 2}{p_{i}})}\lVert \Delta u\rVert_{L^{2}}^{2(\frac{2}{p_{i}})}
\leq \frac{1}{16} Z(t) + c\sum_{i=3}^{4} \lVert \nabla u_{i} \rVert_{L^{p_{i}}}^{\frac{p_{i}}{p_{i} - 2}}X(t)\nonumber
\end{align}
by H$\ddot{o}$lder's, Gagliardo-Nirenberg and Young's inequalities. Next, we decompose
\begin{align}
III_{4} =& -\sum_{i,j=1}^{4}\sum_{k=3}^{4}\int \partial_{k}u_{i} \partial_{i}b_{j}\partial_{k}b_{j} \\
=& -\sum_{i=1}^{2}\sum_{j=1}^{4}\sum_{k=3}^{4}\int \partial_{k}u_{i}\partial_{i}b_{j}\partial_{k}b_{j} - \sum_{j=1}^{4}\sum_{i,k=3}^{4}\int \partial_{k}u_{i}\partial_{i}b_{j}\partial_{k}b_{j}\nonumber\\
\lesssim& \int \lvert \nabla u\rvert \lvert \nabla_{1,2} b\rvert \lvert \nabla b\rvert + \sum_{i=3}^{4} \int \lvert \nabla u_{i} \rvert \lvert \nabla b\rvert^{2} \triangleq V_{3} + V_{4}\nonumber
\end{align}
where we estimate $V_{3}$ as $IV_{3}$ in (47) while same estimate of $V_{2}$ in (64) lead to 
\begin{align}
V_{4} \lesssim& \sum_{i=3}^{4} \lVert \nabla u_{i} \rVert_{L^{p_{i}}} \lVert \nabla b\rVert_{L^{\frac{2p_{i}}{p_{i} - 1}}}^{2}\\
\lesssim& \sum_{i=3}^{4} \lVert \nabla u_{i} \rVert_{L^{p_{i}}} \lVert \nabla b\rVert_{L^{2}}^{2(\frac{p_{i} - 2}{p_{i}})}\lVert \Delta b\rVert_{L^{2}}^{2(\frac{2}{p_{i}})} \leq \frac{1}{16} Z(t) + c\sum_{i=3}^{4} \lVert \nabla u_{i} \rVert_{L^{p_{i}}}^{\frac{p_{i}}{p_{i} - 2}}X(t).\nonumber
\end{align}
Finally, 
\begin{align}
III_{6} + III_{8} =& \sum_{i,j=1}^{4}\sum_{k=3}^{4}\int \partial_{k}b_{i}\partial_{i}b_{j}\partial_{k}u_{j} + \partial_{k}b_{i}\partial_{i}u_{j}\partial_{k}b_{j}\\
\lesssim& \int \lvert \nabla b\rvert^{2} \lvert \nabla u\rvert\nonumber\\
\lesssim& \lVert \nabla b\rVert_{L^{p_{b}}}\lVert \nabla b\rVert_{L^{\frac{2p_{b}}{p_{b} - 1}}}\lVert \nabla u\rVert_{L^{\frac{2p_{b}}{p_{b} - 1}}}\nonumber\\
\lesssim& \lVert \nabla b\rVert_{L^{p_{b}}} \lVert \nabla b\rVert_{L^{2}}^{\frac{p_{b} - 2}{p_{b}}}\lVert \Delta b\rVert_{L^{2}}^{\frac{2}{p_{b}}} \lVert \nabla u\rVert_{L^{2}}^{\frac{p_{b} -2}{p_{b}}} \lVert \Delta u\rVert_{L^{2}}^{\frac{2}{p_{b}}}\nonumber\\
\leq& \frac{1}{16}Z(t) + c\lVert \nabla b\rVert_{L^{p_{b}}}^{\frac{p_{b}}{p_{b} - 2}}X(t) \nonumber
\end{align}
by H$\ddot{o}$lder's, Gagliardo-Nirenberg and Young's inequalities. Thus, we obtain by applying (62)-(67) in (41), absorbing and integrating in time, 
\begin{align}
&X(t) + \frac{3}{2}\int_{0}^{t} Z(\tau) d\tau\\
\lesssim& 1 + \sum_{i=3}^{4} \int_{0}^{t} (\lVert \nabla u_{i} \rVert_{L^{p_{i}}}^{\frac{p_{i}}{p_{i} - 2}} + \lVert \nabla b\rVert_{L^{p_{b}}}^{\frac{p_{b}}{p_{b} - 2}})X(\tau) d\tau\nonumber\\
&+ \sup_{\tau \in [0,t]} W^{\frac{1}{2}}(\tau) \left(\int_{0}^{t} Y(\tau) d\tau\right)^{\frac{1}{2}} \left(\int_{0}^{t} Z(\tau) d\tau\right)^{\frac{1}{2}}\nonumber
\end{align}
where we also used H$\ddot{o}$lder's inequality. Now we assume $p_{i} \in (\frac{12}{5}, 4]$. For the last term only, we bound it by a constant multiples of 
\begin{align*}
&\left(1 + \sum_{i=3}^{4} \int_{0}^{t} \lVert \nabla u_{i} \rVert_{L^{p_{i}}}^{\frac{4p_{i}}{3p_{i} -4}} X^{\frac{4(p_{i} -2)}{3p_{i} -4}}(\tau) Z^{\frac{4-p_{i}}{3p_{i} -4}}(\tau) + \lVert \nabla b\rVert_{L^{p_{b}}}^{\frac{4p_{b}}{3p_{b} -4}} X^{\frac{4(p_{b} -2)}{3p_{b} - 4}}(\tau) Z^{\frac{4-p_{b}}{3p_{b}  -4}}(\tau) d\tau\right)\\
&\times \left(\int_{0}^{t}Z(\tau) d\tau\right)^{\frac{1}{2}}\\
\lesssim& \left(\int_{0}^{t}Z(\tau) d\tau\right)^{\frac{1}{2}} + \sum_{i=3}^{4} \left(\int_{0}^{t} \lVert \nabla u_{i} \rVert_{L^{p_{i}}}^{\frac{p_{i}}{p_{i} - 2}} X(\tau) d\tau\right)^{\frac{4(p_{i} - 2)}{3p_{i} - 4}}\left(\int_{0}^{t} Z(\tau) d\tau \right)^{\frac{4+p_{i}}{2(3p_{i} - 4)}}\\
&+ \left(\int_{0}^{t} \lVert \nabla b\rVert_{L^{p_{b}}}^{\frac{p_{b}}{p_{b} - 2}}X(\tau) d\tau\right)^{\frac{4(p_{b} - 2)}{3p_{b} - 4}}\left(\int_{0}^{t} Z(\tau) d\tau\right)^{\frac{4+p_{b}}{2(3p_{b} - 4)}}\\
\leq& \frac{1}{2}\int_{0}^{t}Z(\tau) d\tau\\
&+ c\left(1+ \sum_{i=3}^{4} \left(\int_{0}^{t} \lVert \nabla u_{i} \rVert_{L^{p_{i}}}^{\frac{p_{i}}{p_{i} - 2}}X(\tau) d\tau\right)^{\frac{8(p_{i} - 2)}{5p_{i} - 12}} + \left(\int_{0}^{t} \lVert \nabla b\rVert_{L^{p_{b}}}^{\frac{p_{b}}{p_{b} - 2}}X(\tau) d\tau \right)^{\frac{8(p_{b} - 2)}{5p_{b} - 12}}\right)\\
\leq& \frac{1}{2}\int_{0}^{t}Z(\tau) d\tau + c\left(1+ \sum_{i=3}^{4} \left(\int_{0}^{t} \lVert \nabla u_{i}\rVert_{L^{p_{i}}}^{\frac{8p_{i}}{5p_{i} - 12}}X(\tau) d\tau\right) + \left(\int_{0}^{t} \lVert \nabla b\rVert_{L^{p_{b}}}^{\frac{8p_{b}}{5p_{b} - 12}}X(\tau) d\tau\right)\right)
\end{align*}
due to Proposition 4.1, H$\ddot{o}$lder's and Young's inequalities and (3). 

Next, we consider the case $4 < p_{i} < \infty$. We restart from (41) where we continue our estimates from (55), (59) and (60) to obtain 
\begin{align}
&III_{1} + III_{3} + III_{5} + III_{7} \\
\lesssim& \sum_{i=3}^{4} \lVert \nabla u_{i} \rVert_{L^{p_{i}}} \lVert \nabla_{1,2} u\rVert_{L^{2}}^{\frac{p_{i} - 4}{p_{i}}}\lVert \nabla\nabla_{1,2} u\rVert_{L^{2}}^{\frac{4}{p_{i}}}\lVert \nabla u\rVert_{L^{2}}\nonumber\\
&+ \lVert \nabla b\rVert_{L^{p_{b}}}\lVert \nabla_{1,2} b\rVert_{L^{2}}^{\frac{p_{b} -4}{p_{b}}} \lVert \nabla\nabla_{1,2} b\rVert_{L^{2}}^{\frac{4}{p_{b}}} \lVert \nabla u\rVert_{L^{2}}\nonumber\\
\leq& \frac{1}{16} Z(t) + c\sum_{i=3}^{4} (\lVert \nabla u_{i} \rVert_{L^{p_{i}}}^{\frac{p_{i}}{p_{i} -2}} + \lVert \nabla b\rVert_{L^{p_{b}}}^{\frac{p_{b}}{p_{b} -2}} )X(t) \nonumber
\end{align}
by Young's inequality. The rest of the estimates of $III_{2}, III_{4}, III_{6}, III_{8}$ all go through as in the case $p_{i} \in [\frac{12}{5}, 4]$. Indeed, continuing from (63), we bound $III_{2} \lesssim V_{1} + V_{2}$ where $V_{1}$ is estimated as $IV_{1}$ in (44) and $V_{2}$ is estimated identically as (64). The estimates of $III_{4}$ also goes through as in (65): $III_{4} \lesssim V_{3} + V_{4}$ where $V_{3}$ is estimated as $IV_{3}$ in (47) and $V_{4}$ in (66). Finally, we use the estimate of $III_{6} + III_{8}$ in (67). Thus, in sum, after absorbing, integrating in time, we obtain 
\begin{align*}
X(t) + \frac{3}{2}\int_{0}^{t}Z(\tau) d\tau 
\leq X(0) &+ c\sum_{i=3}^{4}\int_{0}^{t} (\lVert \nabla u_{i} \rVert_{L^{p_{i}}}^{\frac{p_{i}}{p_{i} -2}} + \lVert \nabla b\rVert_{L^{p_{b}}}^{\frac{p_{b}}{p_{b} - 2}})X(\tau) d\tau \\
&+ c\sup_{\tau \in [0,t]} W^{\frac{1}{2}}(\tau) \left(\int_{0}^{t} Y(\tau) d\tau\right)^{\frac{1}{2}} \left(\int_{0}^{t} Z(\tau) d\tau\right)^{\frac{1}{2}}
\end{align*}
by H$\ddot{o}$lder's inequality. We bound the last term by 
\begin{align*}
&c(W(0) + \sum_{i=3}^{4} \int_{0}^{t} (\lVert \nabla u_{i} \rVert_{L^{p_{i}}}^{\frac{p_{i}}{p_{i} - 2}} + \lVert \nabla b\rVert_{L^{p_{b}}}^{\frac{p_{b}}{p_{b} - 2}})X(\tau) d\tau \left(\int_{0}^{t} Z(\tau) d\tau\right)^{\frac{1}{2}}\\
\leq& \frac{1}{2}\int_{0}^{t} Z(\tau) d\tau + c\left(1+ \sum_{i=3}^{4}\left(\int_{0}^{t} (\lVert \nabla u_{i} \rVert_{L^{p_{i}}}^{\frac{p_{i}}{p_{i} - 2}} + \lVert \nabla b\rVert_{L^{p_{b}}}^{\frac{p_{b}}{p_{b} - 2}})X(\tau) d\tau\right)^{2} \right)\\
\leq& \frac{1}{2}\int_{0}^{t} Z(\tau) d\tau + c\left(1+ \sum_{i=3}^{4}\int_{0}^{t} (\lVert \nabla u_{i} \rVert_{L^{p_{i}}}^{\frac{2p_{i}}{p_{i} - 2}} + \lVert \nabla b\rVert_{L^{p_{b}}}^{\frac{2p_{b}}{p_{b} - 2}})X(\tau) d\tau\right)
\end{align*}
by Proposition 4.1, Young's and H$\ddot{o}$lder's inequalities and (3). After absorbing, Gronwall's inequality implies the desired result. We now consider the case $p_{i} = \infty$. For simplicity, we assume $p_{i} = \infty \forall \hspace{1mm} i = 3, 4, b$. We continue from (41) where we estimate in contrast to (69), 
\begin{align*}
&III_{1} + III_{3} + III_{5} + III_{7}\\
\lesssim& \sum_{i=3}^{4} \int \lvert \nabla u_{i} \rvert \lvert \nabla_{1,2} u\rvert \lvert \nabla u\rvert + \lvert \nabla b\rvert \lvert \nabla_{1,2} b\rvert \lvert \nabla u\rvert \lesssim \sum_{i=3}^{4} (\lVert \nabla u_{i} \rVert_{L^{\infty}} + \lVert \nabla b\rVert_{L^{\infty}})X(t)  
\end{align*}
due to (55), H$\ddot{o}$lder's and Young's inequalities. Moreover, from $III_{2} \lesssim V_{1} + V_{2}$ of (63), we estimate $V_{1}$ is estimated as $IV_{1}$ in (44) and 
\begin{equation*}
V_{2} \approx \sum_{i=3}^{4}\int \lvert \nabla u_{i} \rvert \lvert \nabla u\rvert^{2} \lesssim \sum_{i=3}^{4} \lVert \nabla u_{i} \rVert_{L^{\infty}} \lVert \nabla u\rVert_{L^{2}}^{2}.
\end{equation*}
Moreover, from $III_{4} \lesssim V_{3} + V_{4}$ of (65), we have $V_{3}$ estimated as $IV_{3}$ in (47) while
\begin{equation*}
V_{4} \approx \sum_{i=3}^{4}\int \lvert \nabla u_{i} \rvert \lvert \nabla b\rvert^{2} \lesssim \sum_{i=3}^{4} \lVert \nabla u_{i} \rVert_{L^{\infty}} \lVert \nabla b\rVert_{L^{2}}^{2}.
\end{equation*}
Finally, continuing our estimate from (67), 
\begin{align*}
III_{6} + III_{8} \lesssim \int \lvert \nabla b\rvert^{2} \lvert \nabla u\rvert
\lesssim \lVert \nabla b\rVert_{L^{\infty}}\lVert \nabla b\rVert_{L^{2}} \lVert \nabla u\rVert_{L^{2}}
\lesssim \lVert \nabla b\rVert_{L^{\infty}} X(t).
\end{align*}
In sum, integrating in time we obtain 
\begin{align*}
&X(t) + 2\int_{0}^{t} Z(\tau) d\tau\\
\lesssim& X(0) + \sum_{i=3}^{4}\int (\lVert \nabla u_{i} \rVert_{L^{\infty}} + \lVert \nabla b\rVert_{L^{\infty}})X(\tau) d\tau\\
&+ \sup_{\tau \in [0,t]} W^{\frac{1}{2}}(\tau) \left(\int_{0}^{t} Y(\tau) d\tau\right)^{\frac{1}{2}} \left(\int_{0}^{t} Z(\tau) d\tau\right)^{\frac{1}{2}}\\
\lesssim& X(0)+  \sum_{i=3}^{4}\int (\lVert \nabla u_{i} \rVert_{L^{\infty}} + \lVert \nabla b\rVert_{L^{\infty}})X(\tau) d\tau\\
&+\left(W(0) + \sum_{i=3}^{4} \int_{0}^{t} (\lVert \nabla u_{i} \rVert_{L^{\infty}} + \lVert \nabla b\rVert_{L^{\infty}}) X(\tau) d\tau\right)\left(\int_{0}^{t} Z(\tau) d\tau \right)^{\frac{1}{2}}\\
\leq& \int_{0}^{t} Z(\tau) d\tau + c \left(1+ \sum_{i=3}^{4} (\int_{0}^{t} \lVert \nabla u_{i} \rVert_{L^{\infty}}^{2} + \lVert \nabla b\rVert_{L^{\infty}}^{2})X(\tau) d\tau)\right)
\end{align*}
by H$\ddot{o}$lder's inequality, Proposition 4.1, Young's inequality and (3). 

Finally, we prove the smallness result in the case $p_{i} = \frac{12}{5}, r_{i} = \infty$, for which for simplicity of presentation, we assume $r_{i} = \infty, p_{i} = \frac{12}{5}, \forall \hspace{1mm} i = 3, 4, b$. From (68),
\begin{align*}
&X(t) + \frac{3}{2}\int_{0}^{t}Z(\tau) d\tau\\
\lesssim& X(0) + \sum_{i=3}^{4}\int_{0}^{t} (\lVert \nabla u_{i} \rVert_{L^{\frac{12}{5}}}^{6} + \lVert \nabla b\rVert_{L^{\frac{12}{5}}}^{6})X(\tau) d\tau\\
&+ \left(W(0) + \sum_{i=3}^{4}\int_{0}^{t} (\lVert \nabla u_{i} \rVert_{L^{\frac{12}{5}}}^{3} + \lVert \nabla b\rVert_{L^{\frac{12}{5}}}^{3}) X^{\frac{1}{2}}(\tau) Z^{\frac{1}{2}}(\tau) d\tau \right) \left(\int_{0}^{t} Z(\tau) d\tau\right)^{\frac{1}{2}}\\
\leq& \frac{1}{4}\int_{0}^{t}Z(\tau) d\tau + c\sum_{i=3}^{4}\int_{0}^{t} (\lVert \nabla u_{i} \rVert_{L^{\frac{12}{5}}}^{6} + \lVert \nabla b\rVert_{L^{\frac{12}{5}}}^{6})X(\tau) d\tau\\
&+ c\left(1+ \sum_{i=3}^{4} (\int_{0}^{t} (\lVert \nabla u_{i} \rVert_{L^{\frac{12}{5}}}^{3} + \lVert \nabla b\rVert_{L^{\frac{12}{5}}}^{3}) X^{\frac{1}{2}}(\tau) Z^{\frac{1}{2}}(\tau) d\tau )^{2}\right)\\
\leq& \frac{1}{4} \int_{0}^{t}Z(\tau) d\tau + c\sum_{i=3}^{4}\sup_{\tau \in [0,t]} (\lVert \nabla u_{i}\rVert_{L^{\frac{12}{5}}}^{6} + \lVert \nabla b\rVert_{L^{\frac{12}{5}}}^{6})(\tau) \int_{0}^{t}X(\tau) d\tau\\
&+ c\left(1+ \sum_{i=3}^{4}\sup_{\tau \in [0,t]} (\lVert \nabla u_{i} \rVert_{L^{\frac{12}{5}}}^{6} + \lVert \nabla b\rVert_{L^{\frac{12}{5}}}^{6})(\tau) \int_{0}^{t}X(\tau) d\tau \int_{0}^{t} Z(\tau) d\tau \right)\\
\leq& \frac{1}{2} \int_{0}^{t} Z(\tau) d\tau + c
\end{align*}
for $\sum_{i=3}^{4}\sup_{t \in [0,T]} (\lVert \nabla u_{i} \rVert_{L^{\frac{12}{5}}}^{6} + \lVert \nabla b\rVert_{L^{\frac{12}{5}}}^{6})(t)$ sufficiently small where we used H$\ddot{o}$lder's inequality, Proposition 4.1, Young's inequality, (19) and (3). This completes the proof of Theorem 1.4. 
\end{proof}

\section{Proof of Theorem 1.5}
We fix $q_{i} \in (\frac{12}{7}, 6)$ and then $p_{i} = 6 + \epsilon$ for $\epsilon > 0$ sufficiently small so that  $\frac{2(6+\epsilon)}{(6+\epsilon) + 1} < q_{i}$ and also $q_{i} < 6 < p_{i}$. This implies that $\forall \hspace{1mm} \epsilon > 0$ sufficiently small, we have $q_{i} \in (\frac{2p_{i}}{p_{i} + 1}, p_{i})$. Now we multiply the $i$-th component of (1a) with $\lvert u_{i} \rvert^{p_{i} -2} u_{i}$, integrate in space to obtain 
\begin{align*}
\frac{1}{p_{i}} \partial_{t} \lVert u_{i} \rVert_{L^{p_{i}}}^{p_{i}} + c(p_{i}) \lVert u_{i} \rVert_{L^{2p_{i}}}^{p_{i}}
\lesssim& \lVert \partial_{i} \pi \rVert_{L^{q_{i}}} \lVert u_{i} \rVert_{L^{\frac{(p_{i} - 1) q_{i}}{q_{i} - 1}}}^{p_{i} - 1}\\
\lesssim&  \lVert \partial_{i} \pi \rVert_{L^{q_{i}}} \lVert u_{i} \rVert_{L^{p_{i}}}^{\frac{p_{i}q_{i} - 2p_{i} + q_{i}}{q_{i}}}\lVert u_{i} \rVert_{L^{2p_{i}}}^{\frac{2(p_{i} - q_{i})}{q_{i}}}\nonumber\\
\leq& \frac{c(p_{i})}{2} \lVert u_{i} \rVert_{L^{2p_{i}}}^{p_{i}} + c\lVert \partial_{i} \pi \rVert_{L^{q_{i}}}^{\frac{p_{i}q_{i}}{p_{i}q_{i} - 2p_{i} + 2q_{i}}}\lVert u_{i} \rVert_{L^{p_{i}}}^{p_{i}(\frac{p_{i}q_{i} - 2p_{i} + q_{i}}{p_{i}q_{i} - 2p_{i} + 2q_{i}})}\nonumber
\end{align*}
where we used the lower bound estimate on the dissipative term of 
\begin{align*}
c(p_{i})\lVert u_{i} \rVert_{L^{2p_{i}}}^{p_{i}} \approx& \lVert \lvert u_{i} \rvert^{\frac{p_{i}}{2}}\rVert_{L^{4}}^{2} \lesssim \lVert \lvert u_{i} \rvert^{\frac{p_{i}}{2}} \rVert_{\dot{H}^{1}}^{2} \approx \frac{(p_{i} - 1)4}{p_{i}^{2}}\int \lvert \nabla \lvert u_{i} \rvert^{\frac{p_{i}}{2}} \rvert^{2} = -\int \Delta u \lvert u_{i} \rvert^{p_{i} -2} u_{i}
\end{align*}
for some constant $c(p_{i})$ that depends on $p_{i}$, H$\ddot{o}$lder's, interpolation and Young's inequalities. We absorb and obtain 
\begin{align*}
\frac{1}{p_{i}} \partial_{t} \lVert u_{i} \rVert_{L^{p_{i}}}^{p_{i}} + \frac{c(p_{i})}{2} \lVert u_{i} \rVert_{L^{2p_{i}}}^{p_{i}} \lesssim \lVert \partial_{i} \pi \rVert_{L^{q_{i}}}^{\frac{p_{i}q_{i}}{p_{i}q_{i} - 2p_{i} + 2q_{i}}}(1 + \lVert u_{i} \rVert_{L^{p_{i}}}^{p_{i}})
\end{align*}
by Young's inequality. By hypothesis of Theorem 1.5 and Gronwall's inequality, $\forall \hspace{1mm} \epsilon > 0$ sufficiently small we have $\sum_{i=3}^{4} \sup_{t \in [0,T]} \lVert u_{i} \rVert_{L^{p_{i}}}(t) \lesssim 1$ where $p_{i} = 6 + \epsilon$. By Theorem 1.1, the proof of Theorem 1.5 is complete. 

\section{Further Discussion}
There are many results that exist for the regularity criteria component reduction theory of the three-dimensional NSE and the MHD system that we may look forward to being generalized to the four-dimensional case. We remark however that some of such results did not seem readily generalizable. We also note that in order to reduce our two-component regularity criterion for the four-dimensional NSE to one component or to extend it to higher dimension such as five, it seems to require a new approach. 

\section{Acknowledgment}
The author expresses gratitude to Professor Jiahong Wu and Professor David Ullrich for their teaching and Professor Vladimir Sverak for helpful comments on the presentation of the manuscript.


\begin{thebibliography}{100}  
\addtolength{\leftmargin}{0.2in} 
\setlength{\itemindent}{-0.2in} 

\bibitem[1]{1} J. Beale, T. Kato, A. Majda, \emph{Remarks on breakdown of smooth solutions for the three-dimensional Euler equations}, Comm. Math. Phys., \textbf{94}, 1 (1984), 61-66.

\bibitem[2]{2} H. Bei$\tilde{r}$ao da Veiga, \emph{A new regularity class for the Navier-Stokes equations in $\mathbb{R}^{n}$}, Chin. Ann. Math. Ser. B, \textbf{16} (1995), 407-412.

\bibitem[3]{3} C. Cao, E. S. Titi, \emph{Regularity criteria for the three-dimensional Navier-Stokes equations}, Indiana Univ. Math. J., \textbf{57}, 6 (2008), 2643-2662.

\bibitem[4]{4} C. Cao, E. S. Titi, \emph{Global regularity criterion for the 3D Navier-Stokes equations involving one entry of the velocity gradient tensor}, Arch. Ration. Mech. Anal., \textbf{202}, 3 (2011), 919-932.

\bibitem[5]{5} C. Cao, J. Wu, \emph{Two regularity criteria for the 3D MHD equations}, J. Differential Equations, \textbf{248} (2010), 2263-2274.

\bibitem[6]{6} C. Cao, J. Wu, B. Yuan, \emph{The 2D incompressible magnetohydrodynamics equations with only magnetic diffusion}, SIAM J. Math. Anal., \textbf{1}, 46 (2014), 588-602.

\bibitem[7]{7} J.-Y. Chemin, P. Zhang, \emph{On the critical one component regularity for 3-D Navier-Stokes system}, arXiv:1310.6442 [math.AP]

\bibitem[8]{8} H. Dong, D. Du, \emph{Partial regularity of solutions to the four-dimensional Navier-Stokes equations at the first blow-up time}, Comm. Math. Phys., \textbf{273}, 3 (2007), 785-801. 

\bibitem[9]{9} H. Dong, D. Du, \emph{The Navier-Stokes equations in the critical Lebesgue space}, Comm. Math. Phys., \textbf{292} (2009), 811-827. 

\bibitem[10]{10} H. Dong, R. M. Strain, \emph{On partial regularity of steady-state solutions to the 6D Navier-Stokes equations}, Indiana Univ. Math. J., \textbf{61}, 6 (2012), 2211-2229. 

\bibitem[11]{11} L. Escauriaza, G. Seregin, V. $\check{S}$verak, \emph{$L_{3,\infty}$-solutions of Navier-Stokes equations and backward uniqueness} (In Russian), Usp. Mat. Nauk, \textbf{58}, 2, 350 (2003), 3-44: translation in Russ. Math. Surv., \textbf{58}, 2 (2003), 211-250. 

\bibitem[12]{12} D. Fang, C. Qian, \emph{The regularity criterion for 3D Navier-Stokes equations involving one velocity gradient component}, Nonlinear Anal., \textbf{28} (2013), 86-103.

\bibitem[13]{13} G. P. Galdi, \emph{An introduction to the mathematical theory of the Navier-Stokes equations}, Springer, New York, 2011. 

\bibitem[14]{14} Y. Giga, \emph{Solutions for semilinear parabolic equations in $L^{p}$ and regularity of weak solutions of the Navier-Stokes system}, J. Differential Equations, \textbf{61} (1986), 186-212. 

\bibitem[15]{15} C. He, Z. Xin, \emph{On the regularity of weak solutions to the magnetohydrodynamic equations}, J. Differential Equations, \textbf{213}, 2 (2005), 234-254.

\bibitem[16]{16} X. Jia, Y. Zhou, \emph{A new regularity criterion for the 3D incompressible MHD equations in terms of one component of the gradient of pressure}, J. Math. Anal. Appl., \textbf{396} (2012), 345-350.

\bibitem[17]{17} X. Jia, Y. Zhou, \emph{Regularity criteria for the 3D MHD equations involving partial components}, Nonlinear Anal. Real World Appl., \textbf{13} (2012), 410-418. 

\bibitem[18]{18} T. Kato, \emph{Strong $L_{p}$-solutions of the Navier-Stokes equation in $\mathbb{R}_{m}$, with applications to weak solutions}, Math. Z., \textbf{187} (1984), 471-480. 

\bibitem[19]{19} T. Kato, G. Ponce, \emph{Commutator estimates and the Euler and Navier-Stokes equations}, Comm. Pure Appl. Math., \textbf{41}, 7 (1988), 891-907.

\bibitem[20]{20} I. Kukavica, M. Ziane, \emph{One component regularity for the Navier-Stokes equations}, Nonlinearity, \textbf{19} (2006), 453-460.

\bibitem[21]{21} I. Kukavica, M. Ziane, \emph{Navier-Stokes equations with regularity in one direction}, J. Math. Phys., \textbf{48}, 065203 (2007).

\bibitem[22]{22} J. Leray, \emph{Essai sur le mouvement d'un fluide visqueux emplissant l'espace}, Acta Math., \textbf{63} (1934), 193-248.

\bibitem[23]{23} A. J. Majda, A. L. Bertozzi, \emph{Vorticity and incompressible flow}, Cambridge University Press, Cambridge, 2001.

\bibitem[24]{24} P. Penel, M. Pokorn$\acute{y}$, \emph{On anisotropic regularity criteria for the solutions to 3D Navier-Stokes equations}, J. Math. Fluid Mech., \textbf{13} (2011), 341-353.

\bibitem[25]{25} V. Scheffer, \emph{The Navier-Stokes equations in space dimension four}, Comm. Math. Phys., \textbf{61} (1978), 41-68. 

\bibitem[26]{26} M. Sermange, R. Temam, \emph{Some mathematical questions related to the MHD equations}, Comm. Pure Appl. Math., \textbf{36} (1983), 635-664.
  
\bibitem[27]{27} J. Serrin, \emph{On the interior regularity of weak solutions of the Navier-Stokes equations}, Arch. Ration. Mech. Anal., \textbf{9} (1962), 187-195.  

\bibitem[28]{28} K. Yamazaki, \emph{Remarks on the regularity criteria of three-dimensional magnetohydrodynamics system in terms of two velocity field components}, J. Math. Phys., \textbf{55}, 031505 (2014). 

\bibitem[29]{29} K. Yamazaki, \emph{Component reduction for regularity criteria of the three-dimensional magnetohydrodynamics systems}, Electron. J. Differential Equations, \textbf{2014}, 98 (2014), 1-18.

\bibitem[30]{30} K. Yamazaki, \emph{Regularity criteria of MHD system involving one velocity component and one current density component}, J. Math. Fluid Mech. (2014), DOI 10.1007/s00021-014-0178-1.

\bibitem[31]{31} K. Yamazaki, \emph{Regularity criteria of N-dimensional porous media equation in terms of one partial derivative or pressure scalar field}, Commun. Math. Sci., \textbf{13} (2015), 461-476.

\bibitem[32]{32} K. Yamazaki, \emph{$(N-1)$ velocity components condition for the generalized MHD system in $N-$dimension}, Kinet. Relat. Models, \textbf{7} (2014), 779-792.  

\bibitem[33]{33} K. Yamazaki, \emph{On the three-dimensional magnetohydrodynamics system in scaling-invariant spaces}, arXiv:1409.5174 [math.AP]

\bibitem[34]{34} V. Yudovich, \emph{Non stationary flows of an ideal incompressible fluid}, Zhurnal Vych Matematika, \textbf{3} (1963), 1032-1066. 

\bibitem[35]{35} T. Zhang, \emph{Global regularity for generalized anisotropic Navier-Stokes equations}, J. Math. Phys., \textbf{51}, 123503 (2010) 

\bibitem[36]{36} Y. Zhou, \emph{A new regularity criterion for the Navier-Stokes equations in terms of the gradient of one velocity component}, Methods Appl. Anal., \textbf{9}, 4 (2002), 563-578.

\bibitem[37]{37} Y. Zhou, \emph{Remarks on regularities for the 3D MHD equations}, Discrete Contin. Dyn. Syst. \textbf{12}, 5 (2005), 881-886.

\bibitem[38]{38} Y. Zhou, \emph{On a regularity criterion in terms of the gradient of pressure for the Navier-Stokes equations in $\mathbb{R}^{N}$}, Z. Agnew. Math. Phys., \textbf{57} (2006), 384-392.

\bibitem[39]{39} Y. Zhou, M. Pokorn$\acute{y}$, \emph{On the regularity of the solutions of the Navier-Stokes equations via one velocity component}, Nonlinearity, \textbf{23} (2010), 1097-1107.

\end{thebibliography}
\end{document}